\documentclass[12pt,reqno]{amsart}
\usepackage[utf8]{inputenc}
\usepackage[T1]{fontenc}

\usepackage{enumitem}
\usepackage[oldenum]{paralist}

\usepackage{amssymb}
\usepackage[backend=bibtex, firstinits=true, style=alphabetic, doi=false,sorting=nyt]{biblatex}

\usepackage[colorlinks=true]{hyperref}
\usepackage{cleveref}

\usepackage{tikz}
\usetikzlibrary{calc}

\usepackage{array}
\newcolumntype{x}[1]{>{\raggedright\hspace{0pt}}p{#1}}

\newtheorem{theorem}{Theorem}[section]

\newtheorem{lemma}[theorem]{Lemma}
\newtheorem{corollary}[theorem]{Corollary}
\newtheorem{proposition}[theorem]{Proposition}
\newtheorem{conjecture}[theorem]{Conjecture}

\theoremstyle{definition}
\newtheorem{remark}[theorem]{Remark}
\newtheorem{definition}[theorem]{Definition}
\newtheorem{question}[theorem]{Question}
\newtheorem{example}[theorem]{Example}
\newtheorem{observation}[theorem]{Observation}

\newcommand{\ZZ}{\mathbb{Z}}
\newcommand{\NN}{\mathbb{N}}

\newcommand{\QQ}{\mathbb{Q}}
\newcommand{\KK}{\mathbb{K}}

\newcommand{\set}[1]{\{#1\}}
\newcommand{\with}{\,:\,}
\newcommand{\defa}{:=}
\newcommand{\card}[1]{|#1|}

\newcommand{\qq}[1]{``#1''}

\DeclareMathOperator{\sdepth}{sdepth}
\DeclareMathOperator{\depth}{depth}
\DeclareMathOperator{\spdim}{spdim}
\DeclareMathOperator{\pdim}{pdim}
\DeclareMathOperator{\rk}{rk}

\DeclareMathOperator{\im}{Im}
\DeclareMathOperator{\lcm}{lcm}
\DeclareMathOperator{\Char}{char}

\newcommand{\bool}[1]{\mathfrak{B}(#1)}

\newcommand{\below}[2]{{(#1[}_{#2}}
\renewcommand{\below}[2]{{#2}_{<{#1}}}
\newcommand{\beloweq}[2]{{#2}_{\leq{#1}}}

\newcommand{\del}{\mathrm{del}}
\newcommand{\lk}{\mathrm{lk}}

\newcommand{\exlat}[1][\Delta]{\mathcal{L}(#1)}

\newcommand{\CC}[1]{\mathcal{CC}({#1})}

\newcommand{\I}{\mathrm{I}}
\newcommand{\Q}{\mathrm{Q}}

\begin{document}

\title{Stanley depth and simplicial spanning trees}

\author{Lukas Katth\"an}

\address{Universit\"at Osnabr\"uck, FB Mathematik/Informatik, 49069
Osnabr\"uck, Germany}\email{lukas.katthaen@uos.de}

\subjclass[2010]{Primary: 05E40; Secondary: 13D02,55U10.}

\keywords{Monomial ideal; LCM-lattice; Stanley depth; Stanley conjecture.}
\thanks{The author was partially supported by the German Research Council DFG-GRK~1916.}

\begin{abstract}
We show that for proving the Stanley conjecture,
it is sufficient to consider a very special class of monomial ideals.
These ideals (or rather their lcm lattices) are in bijection with the simplicial spanning trees of skeletons of a simplex.

We apply this result to verify the Stanley conjecture for quotients of monomial ideals with up to six generators.
For seven generators we obtain a partial result.
\end{abstract}

\maketitle

\section{Introduction}
Let $\KK$ be a field, $S$ an $\NN^n$-graded $\KK$-algebra and $M$ a finitely generated $\ZZ^n$-graded $S$-module.
The \emph{Stanley depth} of $M$, denoted $\sdepth M$, is a combinatorial invariant of $M$ related to a conjecture of Stanley from 1982 \cite[Conjecture 5.1]{St} (nowadays called the \emph{Stanley conjecture}), which states that $\depth M \leq \sdepth M$.
We refer the reader to \cite{intro} for an introduction to the subject and to the survey by Herzog \cite{H} for a comprehensive account of the known results.
Most of the research concentrates on the particular case where $S$ is a polynomial ring and $M$ is either a monomial ideal $I \subset S$ or a quotient $S/I$. In the present paper we will also work in this setting.

The main result of the present paper is that for proving the Stanley conjecture for $M = S/I$ or $M = I$, it is sufficient to consider certain very special ideals.
These ideals (more precisely their lcm lattices) are in bijection with certain simplicial complexes which we call \emph{stoss} complexes (short for \emph{Spanning Tree Of a Skeleton of a Simplex}):
\begin{definition}
	A $d$-dimensional simplicial complex on $k$ vertices is called \emph{stoss} complex if it has a complete $(d-1)$-skeleton and is $\KK$-acyclic.
\end{definition}
These complexes can be seen as high dimensional analogues of trees.
They have already been studied by Kalai \cite{kalai} in 1983, and more recently by Duval, Klivans and Martin \cite{DKM} and others.
For $k \in \NN$ we denote by $\bool{k}$ the boolean algebra on $k$ generators, i.e. the set of all subsets of $[k]$ ordered by inclusion.
To every stoss complex $\Delta$ with vertex set $[k]$ we associate a finite lattice $\exlat$ as follows:
	\[ \exlat \defa \set{V \subset [k] \with \Delta|_V \text{ is acyclic}} \subseteq \bool{k},\]
where $\Delta|_V$ denotes the restriction of $\Delta$ to $V$. This is a poset (ordered by inclusion) with maximal element $[k]$ and minimal element $\emptyset$.
Moreover, if $V,W \in \exlat$, then $V \cap W \in \exlat$ (\Cref{lemma:schnittazyklisch}), so $\exlat$ is a finite meet-semilattice, and thus a lattice.

Recall that the \emph{lcm lattice}\cite{GPW} of a monomial ideal $I \subset S$ is the set $L_I$ of all least common multiples (lcm) of subsets of the minimal generators of $I$.
Our main result is the following:
\begin{theorem}\label{thm:main}
	Let $k \geq 2$ and $p \geq 2$.
	Assume that the Stanley conjecture holds for $S/I$ for all ideals $I\subset S$ (in all polynomial rings) with $L_I \cong \exlat$ for some $(p-1)$-dimensional stoss complex $\Delta$ with $k$ vertices. 
	Then the Stanley conjecture holds for $S/I$ for all ideals $I\subset S$ (in all polynomial rings) with $k$ generators and projective dimension $p$.
	
	The same statement holds for $I$ instead of $S/I$.
\end{theorem}
For reasons that will become apparent later, we call lattices of the form $\exlat$ \emph{maximal lattices} and the ideals $I$ with $L_I \cong \exlat$ \emph{extremal ideals}.
It is known from \cite{lcm} that the Stanley projective dimension only depends on the lcm lattice of an ideal, so effectively one only needs to consider one ideal for each stoss complex.
In particular, this reduction is automatically compatible with the reduction of the Stanley conjecture to the squarefree case \cite{IKM}.
On the other hand, it was shown by Herzog, Soleyman Jahan and Zheng \cite{HSJZ} that the Stanley conjecture can be reduced to the case where $S/I$ is Cohen-Macaulay.
The present result is not compatible with this reduction, as extremal ideals are in general not Cohen-Macaulay.

To give an impression of the reduction let us give some explicit numbers. Using the methods described in \cite{IKM3}, we found that there are exactly $7443$ isomorphism types of lcm lattices of ideals with $5$ minimal generators. Of those, $457$ admit a Cohen-Macaulay ideal, but there are only $8$ stoss complexes on $5$ vertices.

We give some properties of extremal ideals.
\begin{theorem}[\Cref{cor:extremal}]
	Let $I \subset S$ be a monomial ideal with $L_I \cong \exlat$ for some $(p-1)$-dimensional stoss complex $\Delta$ on $k$ vertices.
	Then the Scarf complex of $I$ equals $\Delta$, and $I$ has a simplicial minimal free resolution supported on $\Delta$.
	In particular, the projective dimension of $S/I$ is $p$ and its Betti numbers are given by
		\[ \beta^S_{i}(S/I) = 
		\begin{cases}
			\binom{k}{i} &\text{ for } 0 \leq i < p, \\
			\binom{k-1}{i-1} &\text{ for }  i = p, \\
			0 &\text{ for } i > p. \\
		\end{cases}\]
\end{theorem}
It follows that the minimal free resolution of $S/I$ looks like a truncated Koszul complex.
In particular, different extremal ideals with the same number of generators and the same projective dimension have very similar minimal free resolutions.
As the Stanley conjecture can be understood as a question about (multigraded) Hilbert series, and the latter is an alternating sum over the Hilbert series of the modules in the resolution, we hope that this extra structure will prove helpful for further progress.


Let us describe the general idea of the proof.
The following two results by Gasharov, Peeva and Welker, resp. Ichim, the author and Moyano Fern\'andez are the starting point of our considerations.
\begin{theorem}\label{thm:origin}
	Let $I \subset S$ and $I' \subset S'$ be two monomial ideals.
	If there exists a surjective join-preserving map $L_I \rightarrow L_{I'}$, then
	\begin{align*}
	\pdim S'/I' &\leq \pdim S/I, \text{ and}  & \text{\cite{GPW}} \\
	\spdim S'/I' &\leq \spdim S/I.  & \text{\cite{lcm}}
	\end{align*}
	The corresponding statements hold as well for $I$ and $I'$ instead of $S/I$ and $S'/I'$.
\end{theorem}
Here, $\pdim M$ denotes the projective dimension and $\spdim M = n - \sdepth M$ denotes the \emph{Stanley projective dimension}.
Note that the Stanley conjecture can be written in terms of these invariants as $\spdim M \leq \pdim M$.
In view of the result above, the following observation is clear.
\begin{observation}
	To prove the Stanley conjecture for all ideals $I$ or quotients of ideals $S/I$ with a fixed number of generators $k$ and a fixed projective dimension $p$, it is enough to consider those ideals $I$, whose lcm lattice $L_I$ is \emph{not} the image of another lcm lattice of an ideal with the same number of generators and the same projective dimension.
\end{observation}
Consider the (finite) poset of all lcm lattices of monomial ideals with a given number $k$ of generators, ordered by the existence of surjective join-preserving maps.
Then for each $p$ we consider the subposet of those lattices coming from ideals of projective dimension $p$.
Our observation states that we only need to consider the maximal elements of this set.
We will show in \Cref{thm:extremal} and \Cref{thm:reconstruction} that these are exactly the maximal lattices $\exlat$ defined above, 
thus the main result \Cref{thm:main} follows from the observation.


In the second part of the paper, we apply \Cref{thm:main} to verify the Stanley conjecture in some cases.
First, in Sections \ref{sec:amalgamation-of-lattices} and \ref{sec:tools-for-computing-the-stanley-projective-dimension}
we develop techniques for bounding the Stanley projective dimension of maximal lattices.
In particular, in \Cref{lemma:reduction} we obtain a simple way to compute lower bounds of the Stanley depth.

In Section \ref{sec:applications} we apply all our results to verify the Stanley conjecture in several cases. 
First, we show in \Cref{thm:k-2} that if $I$ has $k$ generators and $\spdim S/I = k-2$, then $S/I$ and $I$ satisfy the Stanley conjecture.
This complements several similar results \cite{KSF,HJY,lcm} and in particular implies the Stanley conjecture for ideals with up to five generators.
The latter has also been obtained by different methods in \cite{IKM3}.
Then we turn to the case of six generators.
This is technically more challenging and we can verify the Stanley conjecture only up a single exceptional case.
For this special ideal we resort to a computational proof where we reformulate Stanley decompositions as a system of linear Diophantine equations and inequalities and use the software \texttt{SCIP} \cite{SCIP} to solve it.
Finally, we consider seven generators.
Here we verify the Stanley conjecture for $M = I$. For $M = S/I$, we only obtain a partial result.
This is essentially a computer proof, as we rely heavily on computations for a complete enumeration of cases and the verification of every case.
Of course, this is not very illuminating,
but at least we can narrow down potential counterexample to the Stanley conjecture among the quotients of ideals with seven generators.

In the last section of this paper we discuss various approaches to further research.

\section{Preliminaries}
We fix a field $\KK$ for the whole paper.
Some of the notions defined below depends on the choice of $\KK$, but we will suppress this dependence.
The letter $S$ will always denote a polynomial ring over $\KK$ in $n$ indeterminates.
The number of variables $n$ is \emph{not} fixed, i.e. different occurrences of $S$ might refer to polynomial rings in different numbers of variables.
As the number of variables never enters in our considerations, this should not lead to confusion.

\subsection{Stanley depth and Stanley projective dimension}
Consider the polynomial ring $S$ endowed with the fine $\ZZ^n$-grading.
Let $M$ be a finitely generated graded $S$-module, and let $\lambda$ be a homogeneous element in $M$.
Let $Z \subset \set{X_1, \ldots , X_n}$ be a subset of the set of indeterminates of $S$.
The $\KK[Z]$-submodule $\lambda \KK[Z]$ of $M$ is called a \emph{Stanley space} of $M$ if $\lambda \KK[Z]$ is a free $\KK[Z]$-submodule.
A \emph{Stanley decomposition} of $M$ is a finite family
\[
\mathcal{D}=(\KK[Z_i],\lambda_i)_{i\in \mathcal{I}}
\]
in which $Z_i \subset \{X_1, \ldots ,X_n\}$ and $\lambda_i\KK[Z_i]$ is a Stanley space of $M$ for each $i \in \mathcal{I}$ with
\[
M \cong \bigoplus_{i \in \mathcal{I}} \lambda_i\KK[Z_i]
\]
as a multigraded $\KK$-vector space. This direct sum carries the structure of $S$-module and has therefore a well-defined depth. The \emph{Stanley depth} of $M$, $\sdepth M$, is defined to be the maximal depth of a Stanley decomposition of $M$.
Similarly, the \emph{Stanley projective dimension} $\spdim M$ of $M$ is defined as the minimal projective dimension of a Stanley decomposition of $M$.
Note that $\spdim M + \sdepth M = n$ by the Auslander-Buchsbaum formula.

The Stanley conjecture states that
\[
\sdepth\ M \geq \depth\ M
\]
or equivalently
\[\spdim\ M \leq \pdim\ M. \]

\subsection{Generalities about lattices}
We recall some basic notions from lattice theory.
A \emph{meet-semilattice} $L$ is a partially ordered set (poset), in which every two elements $a,b \in L$ have a unique greatest lower bound $a \wedge b$.
Similarly, a \emph{join-semilattice} $L$ is a poset, in which every two elements $a,b \in L$ have a unique least upper bound $a \vee b$.
Finally, a \emph{lattice} is a poset which is both a meet-semilattice and a join-semilattice.
In the present paper, we consider only finite lattices, hence in the sequel \emph{we assume all lattices to be finite.}
Every finite lattice $L$ has unique minimal and maximal elements, denoted by $\hat{0}_L$ and $\hat{1}_L$.
An \emph{atom} is an element $a \in L$ that covers the minimal element.
A lattice $L$ is called \emph{atomistic}, if every element of $L$ can be written as a join of atoms.
An element $b \in L \setminus\set{\hat{1}_L}$ is called \emph{meet-irreducible}, if it cannot be written as the meet of two strictly greater elements.
Equivalently, an element is meet-irreducible, if it is covered by exactly one other element.
Join-irreducible elements are defines analogously.

For $a \in L$ we use the following notations:
\begin{align*}
	L_{\leq a} &:= \set{b \in L \with b \leq a}, \\
	L_{< a} &:= \set{b \in L \with b < a}.\\
\end{align*}

\section{Lattice theoretical preparations}

\subsection{Lattices in the boolean algebra}
Let $\bool{k}$ denote the lattice of all subsets of $\set{1,2,\dotsc,k}$, i.e. the boolean algebra on $k$ atoms.
In view of \Cref{thm:origin} we are interested in the existence of surjective join-preserving maps between finite lattices.
By the following proposition, we can reduce this to considering inclusions between subsets of $\bool{k}$.
The first part of this proposition was already observed in \cite[p. 5]{M}.
\begin{proposition}\label{prop:embedding}
\begin{enumerate}
	\item For every lattice $L$ with $k$ join-irreducible elements, there exists an embedding $j: L \hookrightarrow \bool{k}$ which respects the meet-operation and the minimal and maximal element.
	\item Let $L, L'$ be two atomistic lattices on $k$ atoms.
	Then there exists a surjective join-preserving map $L \rightarrow L'$ if and only if $j'(L') \subseteq j(L) \subseteq \bool{k}$ for suitable embeddings $j: L \hookrightarrow \bool{k}, j': L' \hookrightarrow \bool{k}$ which respect the meet-operation and the minimal and maximal element.
	\end{enumerate}
\end{proposition}
In the sequel, we will consider all our atomistic lattices as being meet-subsemilattices of $\bool{k}$.
Moreover, we make the following convention for the rest of the paper: \emph{Whenever we have an inclusion of lattices $L' \subseteq L$, we assume that the inclusion respects the meet-operation and the minimal and maximal element.}
With a little extra work, one can show that the embedding $L \hookrightarrow \bool{k}$ is unique up to automorphisms of $\bool{k}$.
If $L$ has less than $k$ join-irreducible elements, then there exists still an embedding, but it is no longer unique.

For the proof of \Cref{prop:embedding} we use the following lattice-theoretic result. It is probably well-known, but as we could not locate a reference we provide a proof.

\begin{lemma}\label{lemma:surjinj}
	Let $L$ and $L'$ be two finite lattices. The following are equivalent:
	\begin{enumerate}
		\item There exists a surjective join-preserving map $\phi: L \twoheadrightarrow L'$ which is injective on the minimal element.
		\item There exists an injective meet-preserving map $j: L' \hookrightarrow L$, such that the minimal element of $L$ is in the image of $j$.
	\end{enumerate}
	In this situation, it holds that $\phi(j(x')) = x'$ and $j(\phi(x))\geq x$ for $x \in L, x' \in L'$.
\end{lemma}
\begin{proof}
\begin{asparaenum}
	\item[(1)$\implies$(2)] Define $j: L' \rightarrow L$ as $j(x') := \bigvee\set{x \in L \with \phi(x) \leq x'}$, cf. \cite[Section 4.1]{lcm}.
		Then 
		\[ j(\phi(x)) = \bigvee\set{y \in L \with \phi(y) \leq \phi(x)}) \geq x\]
		and
		\[ \phi(j(x')) = \phi(\bigvee\set{y \in L \with \phi(y) \leq x'}) = \bigvee\set{y' \in \phi(L) \with y' \leq x'} = x'.\]
		In the last equality we used that $\phi$ is surjective. Hence $j$ is injective.
		
		Further, $j$ is clearly monotonic and thus $j(x' \wedge y') \leq j(x') \wedge j(y')$ for $x', y' \in L'$. On the other hand, 
		\[ j(x' \wedge y') = j(\phi(j(x')) \wedge \phi(j(y'))) \geq j(\phi(j(x') \wedge j(y'))) \geq j(x') \wedge j(y') \]
		thus $j$ preserves the meet. For the middle inequality we use that $\phi$ is monotonic, so $\phi(x) \wedge \phi(y) \geq \phi(x \wedge y)$ for $x,y \in L$.
		It remains to show that $\hat{0}_L$ is in the image of $j$. For this we compute that
		\[ \hat{0}_{L'} \leq \phi(\hat{0}_L) \leq \phi(j(\hat{0}_{L'})) = \hat{0}_{L'}. \]
		As $\phi$ is injective on $\hat{0}_L$, it follows that $j(\hat{0}_{L'}) = \hat{0}_L$.
	\item[(2)$\implies$(1)] Define $\phi: L \rightarrow L'$ as $\phi(x) := \bigwedge\set{x' \in L' \with j(x') \geq x}$.
	One shows analogously that $\phi$ preserves the join, $j(\phi(x)) \geq x$ and $\phi(j(x')) = x'$. In particular, $\phi$ is surjective. Moreover, as $\hat{0}_L$ is in the image of $j$ it holds that $j(\hat{0}_{L'}) = \hat{0}_L$. Hence if $\phi(x) = \hat{0}_{L'}$, then $\hat{0}_L = j(\hat{0}_{L'}) = j(\phi(x)) \geq x$, so $x = \hat{0}_L$.
\end{asparaenum}
\end{proof}

\begin{proof}[Proof of \Cref{prop:embedding}]
	\begin{enumerate}
		\item Let $L$ be a lattice with at most $k$ join-irreducible elements.
		There exists a surjective join-preserving map $\phi: \bool{k} \twoheadrightarrow L$ mapping atoms to join-irreducible elements.
		Hence the map $j: L \rightarrow \bool{k}$ constructed in the preceding lemma gives an embedding of $L$. 
		
		\item This is immediate from part (1) and the preceding lemma.
	\end{enumerate}
\end{proof}

\begin{definition}
	We define the \emph{rank}, $\rk_L a$ of an element $a \in L$ as
	\[ \rk_L a \defa \#\set{b \in L \with b \text{ join-irreducible, } b \leq a} \]
\end{definition}
The rank is the restriction of the usual rank function on $\bool{k}$ to $L \subseteq \bool{k}$.
However, in general it is \emph{not} a rank function on $L$ in the poset-theoretic sense.

\subsection{The lcm lattice}
Let $G \subset S$ be a finite set of monomials.
We write $L_G$ for the lattice of all least common multiples of subsets of $G$, together with a minimal element $\hat{0}$.
For a monomial ideal $I \subseteq S$, we set $L_I := L_G$ for a minimal monomial generating set $G$ of $I$.
Note that $L_G$ is atomistic if and only if $G$ is the \emph{minimal} generating set of the ideal generated by it.

The following theorem recalls the central relation between an ideal and its lcm lattice.
\begin{theorem}[Theorem 2.1, \cite{GPW}]\label{thm:gpw}
	Let $G \subset S$ be a finite set of monomials, $L = L_G$, and $I = (G)$ be the ideal generated by them.
	Let further $i > 0$ and $m \in L$. Then
	\[ \beta_{i,m}^S(S/I) = \dim_\KK \tilde{H}_{i-2}(L_{< m}; \KK) \]
	and $\beta_{i,m}^S(S/I) = 0$ for multidegrees $m \notin L$.
\end{theorem}
Here, $\tilde{H}_{i-2}(L_{< m}; \KK)$ denotes the reduced simplicial homology of the order complex of $L_{< m} \setminus \set{\hat{0}_L}$.
Moreover, $\beta_{i,m}^S(S/I) := \dim_\KK \mathrm{Tor}^S_i(S/I, \KK)_m$ is the multigraded Betti number of $S/I$ over $S$ in degree $m$.
We write $\beta_i^S(S/I) := \sum_{m \in L} \beta_{i,m}^S(S/I)$ for the total Betti numbers of $S/I$.

To simplify the notation, we define $\tilde{H}_{j}(L,m) := \tilde{H}_{j}({L_{< m}}; \KK)$ for $m\in L$.
Moreover, we set
\[ \beta_i(L) := \sum_{m \in L} \dim_\KK \tilde{H}_{i-2}(L,m) \]
and $\mathbf{\beta}(L) := (\beta_{-1}(L), \beta_{0}(L), \dots, \beta_r(L))$,
where $r$ is the rank of the maximal element of $L$.
By \Cref{thm:gpw}, we have $\beta_i(L) = \beta^S_i(S/I)$ for every monomial ideal $I \subseteq S$ with $L \cong L_I$.
We further define the \emph{projective dimension} of $L$ as 
		\[ \pdim_\Q L := \max\set{i\with \beta_i(L) \neq 0}.\]
Equivalently, $\pdim_\Q L$ equals the projective dimension of $S/I$ over $S$ for any monomial ideal $I \subseteq S$ (in some polynomial ring $S$ over $\KK$) with $L = L_I$.
Here, the subscript $\Q$ shall remind us of \qq{quotient}.
For our purposes it turns out to be more convenient to work with \emph{crosscut complexes} instead of the order complex of the lattice.
We recall the definition.
Let $A \subset L$ be the set of atoms of $L$.
The \emph{crosscut complex} of $L$ (with respect to $A$) is the simplicial complex $\CC{L} \subset 2^A$, defined as follows:
A set $E \subseteq A$ lies in $\CC{L}$ if the join $\bigvee\set{a \with a\in E}$ is \emph{not} the maximal element of $L$.
If $L$ is atomistic, then the embedding $L \subseteq \bool{k}$ gives a natural inclusion $L \setminus \set{\hat{1}_L} \subseteq \CC{L}$.

By the following theorem, we can use $\CC{L}$ to compute the homology of $L$.
\begin{theorem}[Crosscut theorem, Theorem 10.8 in \cite{bj}]
The order complex of $L \setminus \set{\hat{0}_L, \hat{1}_L}$ is homotopy equivalent to $\CC{L}$.
\end{theorem}
In fact, in \cite{bj} a much more general version of this theorem is given, but we will only need the variant stated here.

\subsection{The Scarf complex of a lattice}
Recall that the \emph{Scarf complex} $\Delta_I$ of a monomial ideal $I \subset S$ 
with minimal generators $m_1, \dots, m_k$ is the simplicial complex
\[ \Delta_I \defa \set{ \sigma \subset [k] \with m_\sigma = m_\tau \implies \sigma = \tau } \]
where $m_\sigma := \lcm(m_i \with i \in \sigma)$.
The analogous definition for lattices is the following:
\begin{definition}[\cite{M}]
The \emph{Scarf complex} of a lattice $L$ is the subset of those elements of $L$, which can be written as a join of atoms in a unique way.
\end{definition}
Consider an embedding $L \subseteq \bool{k}$ (where $k$ is the number of join-irreducible elements of $L$).
Then the Scarf complex can be identified with the largest simplicial complex contained in $L$.
In particular, an element $a \in L$ lies in the Scarf complex, if and only if $L$ contains every element $b \in \bool{k}$ with $b \leq a$.

It is easy to see that if $L$ is isomorphic to the lcm-lattice of some ideal $I$, then the Scarf complex of $L$ equals the Scarf complex of $I$.

\subsection{Factorization of maps}

The following lemma can be seen as a slight refinement of \cite[Lemma 4.4]{lcm} via \Cref{lemma:surjinj}.
\begin{lemma}\label{lem:ordnung}
Let $L$ be a finite $\wedge$-semilattice and $L' \subseteq L$ be a $\wedge$-subsemilattice. 
Let $\set{a_1, a_2, \dotsc, a_r} = L \setminus L'$ be ordered by decreasing rank, i.e. $\rk_L a_i \geq \rk_L a_{i+1}$ for all $i$.
Then for all $1 \leq j \leq r$, the set $L \setminus \set{a_1, \dots, a_j} \subset L$ is also a $\wedge$-subsemilattice.
\end{lemma}
\begin{proof}
	By induction, it is enough to show that $L \setminus \set{a_1}$ is a meet-subsemilattice of $L$. 
	We claim that $a_1$ cannot be written as a meet of two other elements of $L$.
	To the contrary, if $a_1 = b \wedge c$, then $\rk_L b, \rk_L c > \rk_L a_1$.
	Hence by the choice of $a_1$ we have that $b,c \in L'$ and thus $b \wedge c = a_1 \in L'$, a contradiction.
	
	But if $a_1$ cannot be written as a meet of two other elements of $L$, then $L \setminus \set{a_1}$ is closed under taking the meet and the claim follows.
\end{proof}

\newcommand{\ee}{\mathbf{e}}
\begin{lemma}\label{prop:onestep}
	Let $L$ be a finite lattice
	and let $a \in L$ be a meet-irreducible element which is contained in the Scarf complex of $L$.
	Then it holds that either 
	\begin{align*}
	\beta(L\setminus\set{a}) &= \beta(L) \\
	\text{or}  \qquad \beta(L\setminus\set{a}) & = \beta(L) - \ee_p - \ee_{p+1}
	\end{align*}
	where $p := \rk_L a$ and $\ee_p$ denotes the $p$-th unit vector.
\end{lemma}
\begin{proof}
	As $a$ is meet-irreducible, there exists a unique element $a_+ \in L$ covering $a$, i.e. the meet of all elements strictly greater than $a$.
	Let $L' := L \setminus\set{a}$.
	
	We compute $\tilde{H}_i(L,m)$ for every $m \in L$.
	First, for $m \ngeq a$ we have trivially $\below{m}{L}= \below{m}{L'}$.
	Next, if $m > a_+$ and $a$ is not an atom, then we note that $\CC{\beloweq{m}{L}} = \CC{\beloweq{m}{L'}}$ and hence $\tilde{H}_i(L,m) = \tilde{H}_i(L',m)$.
	On the other hand, if $m > a_+$ and $a$ is an atom, then $a_+$ is an atom of $L'$ and it holds again that $\CC{\beloweq{m}{L}} = \CC{\beloweq{m}{L'}}$ and $\tilde{H}_i(L,m) = \tilde{H}_i(L',m)$ , where the crosscut complexes are taken with respect to the respective sets of atoms.
	
	For $m = a$, our assumption on $a$ implies that $\CC{\beloweq{a}{L}}$ is the boundary of a $(p-1)$-simplex, hence removing $a$ decreases the $(p-2)$-nd homology by one.
	
	Finally, consider the case $m = a_+$.
	It is easy to see that $a$ is a facet of the $\CC{\beloweq{a_+}{L}}$ and that $\CC{\beloweq{a_+}{L'}} = \CC{\beloweq{a_+}{L}} \setminus \set{a}$.
	Now removing a facet from a simplicial complex either decreases the $(p-1)$-st Betti number, or it increases the $(p-2)$-nd one. Note that in the latter case, this cancels the effect from $m = a$.
\end{proof}

\begin{remark}
	Let $L, L'$ and $a$ be as in the previous lemma. Even if $a$ is not in the Scarf complex of $L$, there is still an exact sequence
	\[
	\ldots \rightarrow \tilde{H}_i(L,a) \rightarrow \tilde{H}_i(L',a_+) \rightarrow \tilde{H}_i(L,a_+) \rightarrow \tilde{H}_{i-1}(L, a) \rightarrow \ldots.
	\]
	In fact, this is the Mayer-Vietoris sequence coming from the decomposition
	\[ \Delta(\below{a_+}{L} \setminus \set{\hat{0}_L}) = \Delta(\below{a_+}{L'} \setminus \set{\hat{0}_{L'}}) \cup \Delta(\beloweq{a}{L} \setminus \set{\hat{0}_L}) \]
	with
	\[ \Delta(\below{a}{L} \setminus \set{\hat{0}_L}) = \Delta(\below{a_+}{L'} \setminus \set{\hat{0}_{L'}}) \cap \Delta(\beloweq{a}{L} \setminus \set{\hat{0}_L}). \]
	Here, $\Delta(.)$ denotes the order complex. Note that $\Delta(\beloweq{a}{L} \setminus \set{\hat{0}_L})$ is a cone and hence contractible.
\end{remark}

\section{Maximal lattices}
In this section, we prove our main results.
\begin{definition}
	Let $L$ be a finite atomistic lattice.
	We call $L$ \emph{maximal}, if every other lattice $L'$ with the same number of atoms that maps onto $L$ has a higher projective dimension.

	Note that this notion depends on the underlying field $\KK$.
\end{definition}

\subsection{The Scarf complex of a maximal lattice}
\begin{definition}\label{def:stoss}
	A $d$-dimensional simplicial complex on $k$ vertices is called \emph{stoss} complex if it has a complete $(d-1)$-skeleton and is $\KK$-acyclic.
\end{definition}
The name \emph{stoss} comes from \emph{Spanning Tree Of a Skeleton of a Simplex}.
The following is our first main result.
\begin{theorem}\label{thm:extremal}
	Let $L$ be a maximal lattice on $k$ atoms of projective dimension $p$.
	Then its Scarf complex $\Delta$ is a $(p-1)$-dimensional stoss complex.
\end{theorem}
\begin{proof}
We proceed in several steps.
\begin{asparaenum}
\item First we show that the $(p-2)$-skeleton of $\Delta$ is complete.
	Let $a \in \bool{k} \setminus L$ be an element of minimal rank.
	\Cref{lem:ordnung} implies that $L' := L \cup \set{a}$ is a lattice, and because $L$ is maximal, we have $\pdim_\Q L' \geq p+1$ and thus $\beta_{p+1}(L') \neq \beta_{p+1}(L)$.
	Moreover, the proof of \Cref{lem:ordnung} implies that $a$ is meet-irreducible in $L'$.
	Hence by \Cref{prop:onestep} it follows that $\rk_{L'} a$ is either $p$ or $p+1$.
	So every element of rank at most $p-1$ of $L$ belongs to $\Delta$, and in particular, the $(p-2)$-skeleton of $\Delta$ is complete.
\item Next we consider the Betti numbers of $L$.
	By assumption it holds that $\beta_i(L) = 0$ for $i > p$. Further the first part and \Cref{prop:onestep} imply that $\beta_i(L) = \beta_i(\bool{k}) = \binom{k}{i}$ for $i < p$.
	As the alternating sum of the Betti numbers of $L$ equals the alternating sum of the Betti numbers of $\bool{k}$, a short computation yields that the last Betti number is $\beta_{p}(L) = \binom{k-1}{p-1}$.

\item In this step, we show that $f_{i-1}(\Delta) = \beta_i(L)$ for all $i$.
	Here $f_{i-1}(\Delta)$ denotes the number of elements of rank $i$ in $\Delta$.
	As the $(p-2)$-skeleton of $\Delta$ is complete, we have $f_{i-1}(\Delta) = \binom{k}{i} = \beta_i(L)$ for $i < p$.
	More generally, for every $a \in \Delta$, $\CC{\beloweq{a}{L}}$ is the boundary of a $(\rk_L a - 1)$-simplex, and hence $f_{i-1}(\Delta) \leq \beta_{i}(L)$ for all $i$.
	In particular $f_{i-1}(\Delta) = 0$ for $i > \pdim_\Q L = p$.

	It remains to show that $f_{p-1}(\Delta) = \beta_p(L)$.
	It is clear that $f_{p-1}(\Delta) \leq \beta_p(L)$ and that $f_{p-1}(\Delta)$ equals the number of elements of rank $p$ in $L$.
	For the other inequality we consider the set $E$ of rank $p$ elements of $\bool{k} \setminus L$ and let $L' := L \cup E$.
	Every element of $E$ lies in the Scarf complex of $L'$, hence \Cref{prop:onestep} implies that $\beta_i(L') = \beta_i(L)$ for $i \neq p,p+1$.
	We apply the argument of (2) above to $L'$ to conclude that $\beta_{p+1}(L') = \binom{k-1}{p} = \binom{k}{p} - \beta_p(L)$.
	Hence it is sufficient to show that $\card{E} \leq \beta_{p+1}(L')$.

	First, note that every element of $E$ is meet-irreducible in $L'$ and thus covered by exactly one element of $L$.
	Partition the elements of $E$ into classes depending on which element of $L$ covers them.
	Fix an element $b \in L$ and let $a_1, \dots, a_r \in E$ be the elements that are covered by $b$.
	We are going to show that $\dim \tilde{H}_{p-1}(L',b) = r$, because then summing up over $b$ yields the result.
	Note that the elements $a_1, \dotsc, a_r$ are facets of $\CC{\beloweq{b}{L'}}$ and we have
	\[ \CC{\beloweq{b}{L}} = \CC{\beloweq{b}{L'}} \setminus \set{a_1, \dotsc, a_r}. \]
	Let $A$ and $A'$ be the matrices of the top boundary maps of $\CC{\beloweq{b}{L}}$ and $\CC{\beloweq{b}{L'}}$.
	Note that $A$ is obtained from $A'$ by deleting the columns corresponding to $a_1, \dotsc, a_r$ and that by assumption $A$ is injective.
	Hence, if $\dim \ker A' = \dim \tilde{H}_{p-1}(L',b) < r$, then $\dim \im A' > \dim \im A$.
	So in this case the column of some element, say $a_1$, is not in the image of $A$.
	But then the matrix $A_1$ obtained by appending $a_1$ to $A$ is still injective.
	Consequently, $\pdim_\Q(L \cup \set{a_1}) = \pdim_\Q L$, contradicting the maximality of $L$.
	
\item Finally, we show that $\Delta$ is acyclic.
	For this consider a monomial ideal $I \subset S$ in some polynomial ring $S$, whose lcm lattice $L_I$ is isomorphic to $L$.
	By the preceding considerations, the Betti numbers of $L$ and the face numbers of $\Delta$ coincide.
	As every $(i-1)$-face of $\Delta$ contributes to $\beta_i(L)$, this
	implies that $\beta_{i,a}^S(S/I) = \tilde{H}_{i-2}(L, a) = 0$ for $a \in L \setminus \Delta$ and all $i$.
	In other words, the multigraded Betti numbers of $S/I$ are concentrated on the Scarf complex.
	Hence $S/I$ has a cellular minimal free resolution which is supported on the Scarf complex.
	This in turn implies that $\Delta$ is $\KK$-acyclic, cf. \cite[Prop. 4.5]{millersturm}.
\end{asparaenum}
\end{proof}

We collect some useful by-products of the preceding proof.
\begin{corollary}\label{cor:extremal}
Let $L$ be a maximal lattice with projective dimension $p$ and with $k$ atoms. Let further $I \subset S$ be an ideal such that $L_I \cong L$.
\begin{enumerate}
	\item The Betti numbers of $L$ resp. of $S/I$ are
		\[ \beta_{i}(L) = 
		\begin{cases}
			\binom{k}{i} &\text{ for } 0 \leq i < p, \\
			\binom{k-1}{i-1} &\text{ for }  i = p, \\
			0 &\text{ for } i > p. \\
		\end{cases}\]
	\item Let $\Delta$ be the Scarf complex of $L$. Then $\tilde{H}_i(L, a) = 0$ for $a \in L \setminus \Delta$ and all $i$.
	Equivalently, the minimal free resolution of $S/I$ is supported on the Scarf complex $\Delta$.
\end{enumerate}
\end{corollary}

\subsection{Reconstructing the lattice}
In this section we show how to reconstruct a maximal lattice from its Scarf complex.
To each stoss complex $\Delta$ on the vertex set $[k]$, we associate the following poset
	\[ \exlat \defa \set{V \subset [k] \with \Delta|_V \text{ is acyclic}} \subseteq \bool{k}.\]
Here, $\Delta|_V$ is the restriction of $\Delta$ to $V$ (Recall that $\bool{k}$ is the set of subsets of $[k]$).
The main result is the following.
\begin{theorem}\label{thm:reconstruction}
	For every stoss complex $\Delta$, $\exlat$ is the unique maximal lattice $L$ whose Scarf complex equals $\Delta$.
\end{theorem}
\noindent We prepare two lemmata before we present the proof of the theorem.
\begin{lemma}\label{lemma:schnittazyklisch}
	Let $\Delta$ be a stoss complex and let $V,W$ be subsets of its set of vertices.
	If $\Delta|_V$ and $\Delta|_W$ are acyclic, then so is $\Delta|_{V \cap W}$.
\end{lemma}
\begin{proof}
	Let $d = \dim \Delta$ and let $U$ be the vertex set of $\Delta$.
	Let $C_i(\Delta)$ denote the $i$-chains of $\Delta$, i.e., the vector space spanned by the $i$-faces of $\Delta$, and let $\partial_{i}^\Delta: C_i(\Delta) \rightarrow C_{i-1}(\Delta)$ denote the $i$-th boundary map.
	For subsets $V_1 \subset V_2 \subseteq U$, there are natural inclusions $C_i(\Delta|_{V_1}) \subset C_i(\Delta|_{V_2})$ and boundary maps of the smaller complex are just the restrictions of the boundary maps of the larger complex.
	Under these inclusions it clearly holds that $C_i(\Delta|_{V \cap W}) = C_i(\Delta|_{V}) \cap C_i(\Delta|_{W})$.
	
	For our claim we only need to show that $\tilde{H}_{d-1}(\Delta|_{V \cap W}) = 0$.
	Consider a cycle $a \in \ker \partial_{d-1}^{\Delta|_{V\cap W}}$. As $\Delta$ is acyclic there exists a preimage $b \in C_d(\Delta)$ of $a$. Moreover, this preimage is unique because the $d$-th boundary map is injective.
	
	Now $\Delta|_V$ and $\Delta|_W$ are acyclic, hence $b$ is contained in both $C_d(\Delta|_V)$ and $C_d(\Delta|_W)$.
	It follows that $b \in  C_i(\Delta|_{V}) \cap C_i(\Delta|_{W}) = C_i(\Delta|_{V \cap W})$, so $a$ is a boundary in $\Delta|_{V\cap W}$.
\end{proof}

\begin{lemma}\label{lemma:acyclic}
	Let $L$ be an atomistic lattice and let $\Gamma \subset L$ be a subcomplex of its Scarf complex.
	Then the following are equivalent:
	\begin{enumerate}
		\item $\CC{\beloweq{a}{L}}$ is acyclic for all $a \in L \setminus \Gamma$.
		\item $\Gamma|_a$ is acyclic for all $a \in L$.
	\end{enumerate}
\end{lemma}
Here we set $\Gamma|_a := \set{b \in \Gamma\with b \leq a}$.
As a simplicial complex, this is the restriction of $\Gamma$ to those vertices (atoms) which lie below $a$.
\begin{proof}
	We first note that $\Gamma|_a$ is trivially acyclic for $a \in \Gamma$.
	
	Now assume that $a \in L \setminus \Gamma$.
	Consider the elements $\tau_1, \tau_2, \dotsc, \tau_r$ of $\below{a}{L} \setminus \Gamma$, ordered by increasing rank, i.e. $\rk_L \tau_j \leq \rk_L \tau_{j+1}$.
	Let $L_0 := \Gamma|_a$ and $L_j := \Gamma|_a \cup\set{\tau_1, \dots, \tau_j}$ for $1 \leq j \leq r$.
	Note that $L_r = \below{a}{L}$.
	
	This increasing sequence of posets gives rise to an increasing sequence of subcomplexes of $\CC{\beloweq{a}{L}}$.
	We have that
	\begin{align*}
		\CC{L_j} &= \CC{L_{j-1}} \cup \hat{\tau}_j  \quad \text{ and }\\
		\CC{\beloweq{\tau_j}{L}} &= \CC{L_{j-1}} \cap \hat{\tau}_j.
	\end{align*}
	Here $\hat{\tau}_j$ denotes the full simplex generated by $\tau_j$, which is contractible.
	If the first condition holds and $\CC{\beloweq{\tau_j}{L}}$ is acyclic for all $j$, then a Mayer-Vietoris argument implies that $\CC{\Gamma|_a} = \Gamma|_a$ is acyclic as well.
	
	On the other hand, assume that $\Gamma|_a$ is acyclic for all $a \in L$.
	Fix an $a \in L$.
	We proceed by induction on the number of elements in $\below{a}{L} \setminus \Gamma$. 
	If this set is empty, then $\below{a}{L} = \Gamma|_a$, so the claim holds.
	Otherwise, for each $\tau_i$ this number is smaller than $r$, hence by induction $\CC{\beloweq{\tau_j}{L}}$ is acyclic for all $j$ and by the same Mayer-Vietoris argument we conclude that $\CC{\beloweq{a}{L}}$ is acyclic.
\end{proof}

\begin{proof}[Proof of \Cref{thm:reconstruction}]
	Let $L := \exlat$.
	Then $L$ is a meet-sublattice of $\bool{k}$ by \Cref{lemma:schnittazyklisch}. Moreover, $L$ has a maximal element $[k]$, so it is a lattice.

	First, we show that $\pdim_\Q L = \dim \Delta + 1$.
	Note that $\Delta \subset L$ is a subcomplex of the Scarf complex of $L$, simply because $\Delta$ is itself a simplicial complex.
	So \Cref{lemma:acyclic} implies that $\CC{\beloweq{a}{L}}$ is acyclic for all $a \in L \setminus \Delta$.
	So for computing the projective dimension of $L$, we only need to consider elements in $\Delta$.
	But for every $a \in \Delta$ it holds that $\CC{\beloweq{a}{L}}$ is a $(\rk_L a - 2)$-sphere, hence $\pdim_\Q L = \dim \Delta + 1$. 

	By the definition of maximal lattices, there exists a maximal lattice $L' \subseteq \bool{k}$ such that $L \subseteq L'$ and $\pdim_\Q L = \pdim_\Q L'$.
	It is clear that $\Delta$ is contained in the Scarf complex of $L'$, and by \Cref{thm:extremal} their face numbers coincide.
	So $\Delta$ is the Scarf complex of $L'$.
	Now by \Cref{cor:extremal}, $\CC{\beloweq{a}{L'}}$ is acyclic for all $a \in L' \setminus \Delta$.
	So \Cref{lemma:acyclic} and the definition of $L$ imply that $L' \subseteq L$ and thus $L=L'$.
\end{proof}

\subsection{Supplements}
In this section, we give two examples to illustrate how to pass from a stoss complex $\Delta$ to $\exlat$ and to an ideal with this lcm lattice.
Further, we collect some facts about stoss complexes.
\begin{example}
	Consider the path $\Delta_k$ on $k$ vertices, i.e. the graph with vertex set $[k]$ and edges $\set{i,i+1}$ for $i = 1,\dotsc,k-1$.
	This graph is acyclic (i.e. a tree) and thus a stoss complex. 
	Let us compute $L := \exlat[\Delta_k]$.
	It is easy to see that an induced subgraph $\Delta_k|_V$ is acyclic if and only if $V \subset [k]$ is a set of consecutive integers.
	Hence
	\[ \exlat[\Delta_k] = \set{ \set{a,a+1,\dotsc, b} \subset [k] \with 1 \leq a \leq b \leq k} \subset \bool{k}. \]
	We recall from \cite[Thm 3.4]{lcm} how to find an ideal with this lcm lattice.
	For this, we need to assign a monomial $w(m)$ to each $m \in L$, such that $\gcd(w(m), w(m')) = 1$ if $m$ and $m'$ are incomparable.
	Moreover, we require that $w(m) \neq 1$ if $m$ is meet-irreducible. The last condition is that $w(\hat{1}_L) = 1$.
	Then, for each atom $a$ of $L$, the product of all $w(m)$ for $m \ngeq a$ gives a generator of $I$.
	
	Typically, one can choose $w(m) = 1$ for each $m$ that is not meet-irreducible, and $w(m)$ to be just a variable otherwise.
	In our case, the meet-irreducible elements of $L$ are the sets containing $1$ or $k$.
	To see this, note that $\set{a,\dotsc, b}$ is in general covered by $\set{a-1,\dotsc, b}$ and by $\set{a,\dotsc,b+1}$.
	So if either $a = 1$ or $b = k$, then the element is covered by only one other element and thus meet-irreducible.
	We set
	\[w(m) := \begin{cases}
	X &\text{ if } 1 \in m, \\
	Y &\text{ if } k \in m, \\
	1 &\text{ otherwise.}
	\end{cases}\]
	Then for $i \in [k]$ we have $\prod_{m \ngeq \set{i}} w(m) = X^{i-1}Y^{k-i-1}$.
	Hence the corresponding ideal is $I = (X^{k-1}, X^{k-2}Y, \dotsc, Y^{k-1}) = (X,Y)^{k-1} \subset \KK[X,Y]$.
\end{example}

\begin{example}
	As a second example, consider the graph $G$ with vertex set $[k]$ and edges $\set{1,i}$ for $i = 2, \dotsc, k$ for $k \geq 3$. This is again a stoss complex, so its Alexander dual $\Delta := G^\vee$ is a stoss complex as well, cf. \Cref{prop:stree} below.
	To simplify notation, we write $M^c := [k] \setminus M$ for the complement of a set.
	$\Delta$ contains all subsets of $[k]$ of cardinality at most $k-3$.
	Moreover, $\Delta$ contains $\set{a,b}^c$ if and only if $1 \notin \set{a,b}$.
	The lattice $L := \exlat$ contains in addition the sets $\set{a}^c$, such that $\Delta|_{\set{a}^c}$ is acyclic.
	By Alexander duality, these are exactly those $\set{a}^c$, where $\lk_G a$ is acyclic. Hence, $L$ contains $\set{a}^c$ for $a \neq 1$.
	
	Next we identify the meet-irreducible elements of $L$.
	Clearly, the elements of rank $k-1$ are meet-irreducible and we set $w(\set{a}^c) := X_{aa}$.
	An element $\set{a,b}^c$ of rank $k-2$ is covered by $\set{a}^c$ and $\set{b}^c$. But $L$ contains only those elements with $a,b \neq 1$, so there are no meet-irreducible elements of this rank.
	An element $\set{a,b,c}^c$ of rank $k-3$ is covered by $\set{a,b}^c, \set{b,c}^c$ and $\set{a,c}^c$.
	Thus it is meet-irreducible if and only if it contains $1$.
	We set $w(\set{a,b,1}^c) := X_{ab}$, where we assume $X_{ab} = X_{ba}$.
	There are no meet-irreducible elements of lower rank, because the $(k-3)$-skeleton of $\Delta$ is complete.
	
	We conclude that the corresponding ideal is generated by 
	\begin{align*}
	\prod_{m \ngeq \set{1}} w(m) &= \prod_{2\leq a < b \leq k} X_{ab}\qquad \text{ and}\\
	\prod_{m \ngeq \set{a}} w(m) &= \prod_{b=2}^k X_{ab} \qquad\text{ for } a = 2,\dotsc, k,
	\end{align*}
	inside the polynomial ring $\KK[X_{ab} \with 1\leq a, b\leq k]$ with $X_{ab} = X_{ba}$.
\end{example}

Let us collect some general facts about stoss complexes.
These are partially known, but we consider it convenient to collect them in one place.
\begin{proposition}\label{prop:stree}
	\begin{enumerate}
		\item Every stoss complex $\Delta$ is $\KK$-Cohen-Macaulay and its Stanley-Reisner ring has a linear free resolution.
		\item A simplicial complex is a stoss complex if and only if its Alexander dual is a stoss complex. 
		\item Let $\Delta$ be a $(p-1)$-dimensional stoss complex with $k$ vertices.
			The graded Betti numbers of $\KK[\Delta]$ are given by
		\[ \beta_{i, i + (p-1)}(\KK[\Delta]) = \frac{1}{i+p-1}\cdot\frac{(k-1)!}{(i-1)!(p-1)!(k-i-p)!} \]
		and zero otherwise.
		\item Every $(p-1)$-dimensional stoss complex has exactly $\binom{k-1}{p-1}$ $(p-1)$-faces.
	\end{enumerate}
\end{proposition}
\begin{proof}
	For the first claim we use Hochster's formula \cite[Cor. 5.12]{millersturm} 
	\[ \beta_{i, \sigma}(\KK[\Delta]) = \dim_\KK \tilde{H}^{\#\sigma - i - 1}(\Delta|_\sigma, \KK) \]
	for the Betti numbers of $\KK[\Delta]$, where $\sigma \subset [k]$.
	Fix a $\sigma \subset [k]$.
	First note that $\tilde{H}^j(\Delta|_\sigma, \KK) = 0$ for $j > p-1$ for dimension reasons and for $j < p -3$ and all $\sigma$, because the $(p-2)$-skeleton of $\Delta|_\sigma$ is complete.
	Moreover, $\tilde{H}^{p-1}(\Delta|_\sigma, \KK) = 0$ because the $(p-1)$-boundary map of $\Delta$ is injective (as $\Delta$ is acyclic), so every restriction of it is also injective.
	Hence $\beta_{i, \sigma}(\KK[\Delta]) \neq 0$ is only possible for $\#\sigma - i - 1 = p-2$, i.e. $\#\sigma = i+ (p - 1)$.
	So $\KK[\Delta]$ has a $(p-1)$-linear resolution.
	Now the Eagon-Reiner Theorem \cite[Thm. 5.56]{millersturm} implies that $\Delta$ is Cohen-Macaulay, because by the second claim the Alexander
	dual of $\Delta$ is a stoss complex as well and thus has a linear resolution.
	One can also directly see that $\Delta$ is Cohen-Macaulay. For this, we need to prove that $\beta_{i, \sigma}(\KK[\Delta]) = 0$ for $i = k - (p-1)$ and all $\sigma$.
	But by the foregoing, we only need to consider the case $\#\sigma =i + (p-1) = k$, hence $\sigma = [k]$, so the claim follows from the acyclicity of $\Delta$.

	The second claim is \cite[Theorem 5]{kalai}. Alternatively, it is easy to see that the conditions are invariant under Alexander duality.
	The third claim follows from Theorem 4.1.15 in \cite{BH} and the last claim is Proposition 2 in \cite{kalai}.
\end{proof}
The third part of the preceding proposition is interesting in our context for the following reason.
Let $L$ be the maximal lattice corresponding to a stoss complex $\Delta$. 
Then, by Hochster's formula and \Cref{thm:reconstruction}, the multidegrees of the nonzero multigraded Betti numbers of $\Delta$ form exactly the set $\bool{k} \setminus L$.
On the other hand, by the preceding proposition the \emph{$\ZZ$-graded} Betti numbers of $\Delta$ are already determined by the values of $k$ and $p$, so they do not contain much information about $L$.

The second part of the proposition has an interesting consequence:
\begin{corollary}
	Let $k > 3$ and $1\leq p \leq k-1$.
	The number of maximal lattices on $k$ atoms with projective dimension $p$ equals the number of maximal lattices of projective dimension $k-p$ (on the same number of atoms).
\end{corollary}
It would certainly be interesting to understand the relation between a maximal lattice and its \qq{dual} with respect to the Stanley conjecture.
We only make some observations. As stoss complexes have linear resolutions, one can read off the multidegrees of their non-zero Betti numbers from the numerator of its Hilbert series (called $\mathcal{K}$-polynomial in \cite{millersturm}).
On the other hand, Theorem 5.14 in \cite{millersturm} gives a formula for the $\mathcal{K}$-polynomial of the Alexander dual.
So there is a (subtle) combinatorial description of the dual of an maximal lattice.


\section{Amalgamation of Lattices}\label{sec:amalgamation-of-lattices}
In this section we consider the \emph{amalgamation} of two lattices, a construction for lattices that we will use for our applications.
\begin{definition}
	Let $L_2 \subseteq L_1$ be two finite lattices.
	The \emph{amalgamation} of $L_1$ and $L_2$ is the set
	\[ L_1 \# L_2 \defa L_1 \times \set{0}  \cup L_2 \times \set{1} \]
	with the order $(a,i) \leq (b,j)$ if $a \leq b$ in $L_1$ and $i \leq j$ (where $0 < 1$).
\end{definition}

\begin{remark}\label{rem:amal}
	By considering an embedding $L_2 \subseteq L_1 \subseteq \bool{k}$, we may identify $L_1 \# L_2$ with the lattice
	\[ L_1 \# L_2 \cong L_1 \cup \set{a \vee \set{k+1} \with a \in L_2} \subset \bool{k+1}. \]
\end{remark}

We give an easy criterion to recognize amalgamations.
\begin{lemma}\label{prop:zerlegung}
	Let $L$ be a finite atomistic lattice on $k$ atoms.
	Assume that $L$ contains an element $m$ of rank $k-1$.
	Then $L$ is (isomorphic to) an amalgamation $L_1 \# L_2$ of two lattices.
\end{lemma}
\begin{proof}
	Consider an embedding $L \subseteq \bool{k}$.
	Let $L_1 := L_{\leq m}$.
	We may assume that $\set{k}$ is the unique atom which is not below $m$.
	Then every element of $L$ which is not below $m$ contains $k$ and we set $L_2 := \set{a \setminus \set{k} \with a \in L, k \in a}$.
	As $L$ is a meet-subsemilattice of $\bool{k}$, we have that $a \setminus \set{k} = a \wedge m$ if $k \in a \in L$, so $L_2$ is contained in $L_1$.
	Now it follows that
	\[ L = L_1 \cup \set{ a \vee \set{k} \with a \in L_2} = L_1 \# L_2. \]
\end{proof}

In the next proposition we identify the amalgamations among the maximal lattices.
This will allow us frequently to break a maximal lattice into smaller parts.
For a simplicial complex $\Delta$ and a vertex $v$, 
we denote by $\del_\Delta v := \set{F \in \Delta \with v \notin F}$ the \emph{deletion} of $v$,
and by $\lk_\Delta v := \set{F \in \Delta \with v \notin F, F \cup\set{v} \in \Delta}$ the \emph{link} of $v$.
\begin{proposition}\label{prop:zerlegungstoss}
	Let $\Delta$ be a $(p-1)$-dimensional stoss complex on $k > p$ vertices.
	\begin{enumerate}
		\item Every vertex of $\Delta$ is contained in at least $\binom{k-2}{p-2}$ facets.
		\item For a vertex $v$ of $\Delta$, the following statements are equivalent:
			\begin{enumerate}
				\item $v$ is contained in exactly $\binom{k-2}{p-2}$ facets, i.e. the minimum possible value.
				\item $\exlat$ is an amalgamation $L_1 \# L_2$ and $v$ is the (unique) atom of $\exlat$ which is not contained in $L_1$.
			\end{enumerate}
			In this case, $\del_\Delta v$ and $\lk_\Delta v$ are both stoss complexes as well, and $L_1 =  \exlat[\del_\Delta v]$ and $L_2 = \exlat[\lk_\Delta v] \cap \exlat[\del_\Delta v]$.			
		\item If $v$ is a vertex satisfying these conditions, then $v$ also satisfies the conditions for the Alexander dual $\Delta^\vee$, i.e. it is contained in the minimum possible number of facets of $\Delta^\vee$.
	\end{enumerate}
\end{proposition}
\begin{proof}
\begin{asparaenum}
\item
Fix a vertex $v$ of $\Delta$. We consider the Mayer-Vietoris sequence
\[ \dotsb \rightarrow \tilde{H}_{i+1}(\Delta, \KK) \rightarrow \tilde{H}_{i}(\lk_\Delta v, \KK) \rightarrow  \tilde{H}_{i}(\del_\Delta v, \KK) \rightarrow  \tilde{H}_{i}(\Delta, \KK) \rightarrow \dotsb \]
which is induced by the covering $\Delta = \del_\Delta v \cup (v * \lk_\Delta v)$ of $\Delta$ with $\del_\Delta v \cap (v * \lk_\Delta v) = \lk_\Delta v$.
As $\Delta$ is acyclic, we see that $\tilde{H}_{i}(\lk_\Delta v, \KK) \cong \tilde{H}_{i}(\del_\Delta v, \KK)$ for every $i$.
Moreover, $\Delta$ is Cohen-Macaulay, so Reisner's criterion (\cite[Theorem 8.1.6]{HH}) implies that $\tilde{H}_{i}(\lk_\Delta v, \KK) = \tilde{H}_{i}(\del_\Delta v, \KK)=0$ for $i\neq p-2$.
In particular, $\tilde{H}_{p-3}(\lk_\Delta v, \KK) = 0$, thus the $(p-2)$-boundary map of $\lk_\Delta v$ maps surjectively onto the kernel of the $(p-3)$-rd one.
But the $(p-3)$-skeleton of $\lk_\Delta v$ is complete, so the dimension of this kernel is $\binom{k-2}{p-2}$.
Hence $\lk_\Delta v$ has at least as many facets.

\item 
For the second claim, we first assume that $\exlat = L_1 \# L_2$ is an amalgamation.
Let $m$ be the maximal element of $L_1$. By assumption, it is the join of the $k-1$ atoms of $\exlat$ different from $v$.
As $m \in \exlat$, it holds that $\Delta|_{[k]\setminus \set{v}} = \del_\Delta v$ is acyclic and thus $\del_\Delta v$ is a stoss complex.
So it has $\binom{k-2}{p-1}$ facets. Hence $v$ is contained in exactly $\binom{k-1}{p-1} - \binom{k-2}{p-1} = \binom{k-2}{p-2}$ facets of $\Delta$.

For the converse, assume that $\lk_\Delta v$ has exactly $\binom{k-2}{p-2}$ facets.
Then Proposition 2 of \cite{kalai} implies that $\tilde{H}_{p-2}(\lk_\Delta, \KK) = 0$, because $\tilde{H}_{p-3}(\lk_\Delta, \KK) = 0$ and its $(p-3)$-skeleton is complete.
So $\lk_\Delta v$ and, hence, $\del_\Delta v$ are acyclic. Thus both complexes are stoss complexes.
Since $\del_\Delta v = \Delta|_{[k]\setminus{v}}$ is acyclic, there is an element $m \in \exlat$ of rank $k-1$ corresponding to $[k]\setminus \set{v}$.
So by \Cref{prop:zerlegung} we have that $\exlat = L_1 \# L_2$ with $L_1 = \exlat_{\leq m} = \exlat[\del_\Delta v]$.
The last equality follows easily from \Cref{thm:reconstruction}.
In particular, $v$ is the (unique) atom of $\exlat$ which is not contained in $L_1$.

So it remains to compute $L_2$.
By the construction of $\exlat$ and \Cref{rem:amal}, the elements of $L_2$ correspond to subsets $U$ of the vertex set of $\Delta$ with $v \notin U$, such that $\Delta|_{U \cup \set{v}}$ is acyclic.
Further, as this is a restriction of a $(p-1)$-dimensional stoss complex, this is equivalent to the vanishing of $\tilde{H}_{p-2}(\Delta|_{U\cup\set{v}}, \KK)$.

We compute this by considering the Mayer-Vietoris sequence associated to the covering $\Delta|_{U\cup \set{v}} = (\del_\Delta v)|_U \cup (v * \lk_\Delta v)|_U$.
We have $\tilde{H}_{p-2}((\lk_\Delta v)|_U, \KK) = 0$, because the top boundary map is a restriction of the top boundary map of the stoss complex $\lk_\Delta v$ and hence injective.
Moreover, as $\del_\Delta v$ has a complete $(p-2)$-skeleton, the same holds for $(\del_\Delta v)|_U$ and hence $\tilde{H}_{p-3}((\del_\Delta v)|_U, \KK) = 0$. So the Mayer-Vietoris sequence breaks into a short exact sequence
\[ 0 \rightarrow \tilde{H}_{p-2}((\del_\Delta v)|_U, \KK) \rightarrow \tilde{H}_{p-2}(\Delta|_{U\cup\set{v}}, \KK)\rightarrow \tilde{H}_{p-3}((\lk_\Delta v)|_U, \KK)\rightarrow 0. \] 
The term in the middle vanishes if and only if both other terms vanish.
By \Cref{thm:reconstruction}, this tells us that $L_2 = \exlat[\lk_\Delta v] \cap \exlat[\del_\Delta v]$, as claimed.

\item The last claim is a straight-forward computation.
Recall that the number of facets of stoss complexes is fixed.
So if the number of facets containing $v$ is minimal, then the number of facets of $\Delta$ that do not contain $v$ is maximal.
Hence the number of $(p-1)$-dimension \emph{non}-faces not containing $v$ is minimal, and this is just the number of facets of the Alexander dual containing $v$.
\end{asparaenum}
\end{proof}

As a partial converse of the preceding result, we show that the amalgamation of maximal lattices is frequently again maximal.
\begin{proposition}
	Let $L_1$ and $L_2$ be maximal lattices with $k$ atoms of projective dimension $p$ resp. $p-1$.
	Then $L_1 \# (L_1 \cap L_2)$ is again a maximal lattice (with $k+1$ atoms) of projective dimension $p$.
\end{proposition}
\begin{proof}
	Let $\Delta_1$ and $\Delta_2$ denote the Scarf complexes of $L_1$ and $L_2$ and set $\Delta := \Delta_1 \cup v * \Delta_2$, where $v$ is a new vertex and $*$ denotes the join.
	It is easy to see that $\Delta$ has a complete $(p-2)$-skeleton.
	Moreover, note that $\Delta_1, v * \Delta_2$ and $\Delta_1 \cap v * \Delta_2 = \Delta_2$ are all acyclic, 
	so a Mayer-Vietoris argument shows that $\Delta$ is acyclic as well.

	So $\Delta$ is a stoss complex and we may consider the maximal lattice $\exlat$.
	On the other hand, by the preceding proof we have that
	\begin{align*}
	\exlat &= \exlat[\del_\Delta v] \# (\exlat[\lk_\Delta v] \cap \exlat[\del_\Delta v]) \\
	&= \exlat[\Delta_1] \# (\exlat[\Delta_2] \cap \exlat[\Delta_1]) = L_1 \# (L_2 \cap L_1),
	\end{align*}
	so the claim follows.
\end{proof}

We close this section with a general formula for the Betti numbers of an amalgamation.
We will not use it in the sequel (because it is obvious for maximal lattices), but we consider it to be of independent interest, in particular in view of \Cref{conj:amal} below.
\begin{theorem}\label{thm:amalpdim}
	Let $L_2 \subseteq L_1$ be two finite lattices. Then
	\[ \beta_i(L_1 \# L_2) = \beta_i(L_1) + \beta_{i-1}(L_2). \]
	In particular, $\pdim_\Q L_1 \# L_2 = \max\set{\pdim_\Q L_1, \pdim_\Q L_2 + 1}$.
\end{theorem}
\begin{proof}
	Let $L := L_1 \# L_2$.
	For every $a \in L_1$ we have $L_{\leq (a,0)} \cong (L_1)_{\leq a}$ and thus $\tilde{H}_i(L,(a,0)) \cong\tilde{H}_{i}(L_1,a)$.
	
	On the other hand, for $a \in L_2$ it holds that $L_{\leq (a,1)} = (L_1)_{\leq a} \# (L_2)_{\leq a}$ and hence, by \Cref{lemma:suspension} below, $\tilde{H}_i(L,(a,1)) \cong\tilde{H}_{i-1}(L_2,a)$.
	Summing over all $a$ yields the claim.
\end{proof}

\begin{lemma}\label{lemma:suspension}
	Let $L_2 \subseteq L_1$ be two finite lattices. Then
	\[ \CC{L_1 \# L_2} \simeq \mathrm{susp}(\CC{L_2}), \]
	where $\mathrm{susp}(.)$ denotes the suspension.
	In particular, $\tilde{H}_i(L_1 \# L_2,\hat{1}_{L_1 \# L_2}) \cong\tilde{H}_{i-1}(L_2,\hat{1}_{L_2})$ for all $i$.
\end{lemma}
\begin{proof}
	Let $L := L_1 \# L_2$ and let $\Delta := \CC{L}$. 
	Let $v$ be the vertex of $\Delta$ corresponding to the minimal element of $L_2$ inside $L$.
	Then every subset of atoms not including $v$ is bounded above in $L$ by the maximal element of $L_1$.
	Hence $\del_\Delta v$ is the full simplex on all atoms but $v$.
	On the other hand, let $U$ be a set of atoms containing $v$.
	Then $U$ is bounded above if and only if $U \setminus \set{v}$ is bounded above in $L_2$.
	Hence $\lk_\Delta v = \CC{L_2}$.
	In conclusion, 
	\[ \Delta = \del_\Delta v \cup (v * \lk_\Delta v), \]
	both parts are contractible, and the intersection equals $\CC{L_2}$.
	Then it follows from \cite[Lemma 10.4 (i)]{bj} that
	$\Delta \simeq \mathrm{susp}(\CC{L_2})$
	and the proof is complete.
\end{proof}

\section{Tools for computing the Stanley projective dimension}\label{sec:tools-for-computing-the-stanley-projective-dimension}
In this section, we give two useful lemmata (\Cref{lemma:amal} and \Cref{lemma:reduction}) for the computation of the Stanley projective dimension of a lattice.
\begin{definition}
Let $L$ be a finite lattice and let $G \subsetneq S$ be a set of monomials with $L = L_G$.
Let $I := (G)$ the ideal generated by $G$.
We define
\begin{align*}
\spdim_\Q L &:= \spdim S/I \text{ and} \\
\spdim_\I L &:= \spdim I. \\
\end{align*}
Here, the subscripts $\Q$ and $\I$ stand for \qq{quotient} and \qq{ideal}. In particular, the subscript $\I$ is not the name of the ideal $I$ involved in the definition.
\end{definition}
If $L_G$ is not atomistic, then its atoms correspond to the minimal generators of $(G)$.
Hence it holds that $\spdim_{\Q/\I} L = \spdim_{\Q/\I} L^{\text{Atom}}$, where $L^{\text{Atom}}$ is the atomistic sublattice of $L$.
By \Cref{thm:origin} (cf. \cite[Thm 4.5]{lcm}), the invariants $\spdim_\Q L$ and $\spdim_\I L$ do not depend on the choice of $G$.
The following is a useful special case of \cite[Thm 4.5]{lcm}).
\begin{proposition}
Let $L$ and $L'$ be two finite lattices.
Assume that there exists an injective meet-preserving map $j: L' \rightarrow L$, such that the minimal element $\hat{0}_L$ of $L$ lies in the image of $j$.
Then $\spdim_\epsilon L' \leq \spdim_\epsilon L$ for $\epsilon = \I,\Q$.
\end{proposition}
\begin{proof}
	The assumption is equivalent to the existence of a surjective join-preserving map $\phi: L \rightarrow L'$ which is injective on $\hat{0}_L$ by \Cref{lemma:surjinj}, so the claim is immediate from \cite[Thm 4.5]{lcm}\footnote{The assumption in \cite{lcm} that $\phi$ is a map of the lcm-\emph{semi}lattices is equivalent to the assumption that $\phi$ is injective on $\hat{0}_L$, because the semilattices in our situation are $L \setminus \set{\hat{0}_L}$ and $L' \setminus \set{\hat{0}_{L'}}$.}.
\end{proof}
%
We give an estimate for the Stanley projective dimension of an amalgamation. See \Cref{conj:amal} below for a conjectured sharpening of this.
\begin{lemma}\label{lemma:amal}
	Let $L_2 \subseteq L_1$ be two atomistic lattices.
	Then it holds that
	\[ \max\set{\spdim_\Q L_1, \spdim_\Q L_2 + 1} \leq \spdim_\Q L_1 \# L_2 \leq \spdim_\Q L_1 + 1 \]
	and
	\[\spdim_\I L_1 \leq \spdim_\I L_1 \# L_2 \leq \spdim_\I L_1 + 1. \]
\end{lemma}
\begin{proof}
First, we consider $L \# L$ for an atomistic lattice $L$.
Choose an ideal $I \subsetneq S$ such that $L_I = L$.
Let $J := (I,x) \subsetneq S[x]$, where $x$ is a new variable.
It is easy to see that
\[
L_J = L_I \cup \set{a \vee x \with a \in L_I} \cong L \# L.
\]

Then $S/I \cong S[x]/(I,x)$, so the two modules have the same Stanley depth.
Hence $\spdim_\Q L \# L = \spdim_\Q L + 1$.

Further, $J = I \oplus x S[x]$ as vector space, hence $\sdepth_{S[x]} J \geq \sdepth_S I$. It follows that $\spdim_\I L\# L \leq \spdim_\I L + 1$.

Now we turn to the actual proof of the lemma. As we have obvious inclusions
	\[ L_1 \times \set{0} \subset L_1 \# L_2 \subset L_1 \# L_1 \]
as $\wedge$-subsemilattices, we obtain
	\[ \spdim_\epsilon L_1 \leq \spdim_\epsilon L_1 \# L_2 \leq \spdim_\epsilon L_1 \# L_1 \leq \spdim_\epsilon L_1 + 1 \]
for $\epsilon = \I,\Q$. 
Moreover, $L_2 \# L_2 \subset L_1 \# L_2$, hence we have $\spdim_\Q L_2 + 1 \leq \spdim_\Q L_1 \# L_2$.
\end{proof}

The next lemma gives a bound to the Stanley projective dimension in terms of a certain decomposition of $L$.
It is the key to most of our computations.
\begin{lemma}\label{lemma:reduction}
	Let $p \in \NN$, $L$ be a finite atomistic lattice and $a \in L$ meet-irreducible.
	If $\spdim_\epsilon L_{\leq a} < p$, then $\spdim_\epsilon L \leq \max\set{p, \spdim_\epsilon L \setminus \set{a}}$ for $\epsilon=\Q,\I$.
	In particular, the condition $\spdim_\epsilon L_{\leq a} < p$ is satisfied if
	\begin{enumerate}
		\item either $\rk a < p$, or
		\item $\rk a = p$ and $a$ is not contained in the Scarf complex of $L$.
	\end{enumerate}	
	For $\epsilon = \I$, it is sufficient to require $\rk a < 2p$.
\end{lemma}
\begin{proof}
	We construct a monomial ideal $I \subset S$, such that the lcm lattice of $I$ equals $L$ and the lcm lattice of $I:x$ equals $L \setminus \set{a}$ for a certain variable $x \in S$.
	Explicitly, let $E_1 \subset L$ be the set of meet-irreducible elements, let $E_2 \subset L$ be the set of elements covered by $a$ and set $E := E_1 \cup E_2$.
	Then we choose $S = \KK[x_e \with e \in E]$ and let $I \subset S$ be the ideal generated by 
	\[\prod_{\substack{e \in E\\e \ngeq b}} x_e, \]
	where $b$ runs over the atoms of $L$.
	Theorem 3.4 in \cite{lcm} resp. Theorem 3.2 in \cite{M} imply that $L_I \cong L$ and that $L_{I : x_a} \cong L \setminus \set{a}$.
	
	Let moreover $S' := \KK[x_e \with e \in E \setminus \set{a}] \subset S$ and let $I' := I \cap S'$.
	The lcm lattice of $I' := I \cap S'$ is just $L_{\leq a}$, as $I'$ is generated by all monomials in $I$ dividing the monomial $m \in L_I$ corresponding to $a \in L$.
	By considering the vector space decompositions $S/I = S'/I' \oplus x_a S/(I:x_a)$ and $I = I' \oplus x_a (I:x_a)$, one easily obtains then
	\[ \spdim_\epsilon L \leq \max\set{ \spdim_\epsilon L_{\leq a} + 1, \spdim_\epsilon L \setminus \set{a}}, \]
	and the first part of the lemma follows.

	For the `in particular' part, by \cite[Proposition 5.2]{lcm} we have that $\spdim_\Q L_{\leq a} \leq \rk a$ and $\spdim_\I L_{\leq a} \leq \frac{\rk a}{2}$ for all $a \in L$.
	Moreover, equality for $\spdim_\Q L$ holds if and only if $L_{\leq a}$ is a boolean lattice (cf. \cite[Theorem 5.3]{lcm}), i.e. if and only if $a$ is in the Scarf complex of $L$.
\end{proof}
The assumption that $L$ is atomistic is not essential. For a non-atomistic lattice, one can consider a set of monomials $G$ with $L = L_G$ instead of an ideal $I$.
Finally, we observe that we need to consider only one inequality for the computation of the Stanley projective dimension:
\begin{proposition}
	For every maximal lattice $L$ it holds that \[\spdim_\Q L \geq \pdim_\Q L.\]
	In particular, if the Stanley conjecture is true then $\spdim_\Q L = \pdim_\Q L$.
\end{proposition}
\begin{proof}
	If $\pdim_\Q L = p$, then $L$ has an element $a$ of rank $p$ in its Scarf complex.
	Hence $\spdim_\Q L \geq \spdim_\Q L_{\leq a} = \spdim_\Q \bool{p} = p$.
\end{proof}

\begin{example}\label{ex:tree}
Let us illustrate the use of these lemmas by computing the Stanley projective dimension of a maximal lattice $L$ with projective dimension $2$.
This was done in \cite[Lemma 4.3]{KSF} by different methods.

If $\pdim_\Q L = 2$, then the Scarf complex $\Delta$ of $L$ is a $1$-dimensional stoss complex, i.e. a tree.
The elements of $L$ correspond to acyclic induced subcomplexes of $\Delta$, i.e. sets of vertices such that the induced subgraph is connected.
Now if $v$ is a leaf of $\Delta$, then it is easy to see that every element $a > v$ in $L$ has to be greater than the unique edge attached to $v$.
Hence $v$ is meet-irreducible. But $v$ has rank $1$, so $\spdim_\Q L \leq \max\set{2, \spdim_\Q L \setminus \set{v}}$ and $\spdim_\I L \leq \max\set{1, \spdim_\I L \setminus \set{v}}$.
This way we may remove every leaf of $\Delta$.
Further, we may restrict the resulting lattice to its atomistic sublattice.
This removes all edges to the former leaves, resulting in a smaller maximal lattice $L'$.
We iterate this procedure until we are left with a single edge, corresponding to an ideal generated by two indeterminates.
So we conclude that $\spdim_\Q L \leq 2$ and $\spdim_\Q L \leq 1$.
\end{example}

As another simple application, we give the following observation:
\begin{proposition}
	Assume that the Stanley conjecture does not hold.
	Let $I \subset S$ be an ideal with the minimal possible number of generators such that $\spdim S/I > \pdim S/I$.
	Then $\spdim S/I = \pdim S/I + 1$.
\end{proposition}
\begin{proof}
	Let $L$ be the lcm lattice $L_I$ of $I$.
	By the choice of $I$, every lower interval $L_{\leq a}$ for $a \in L \setminus\set{\hat{1}}$ satisfies the Stanley conjecture.
	So we have that
	\[ \spdim_\Q  \beloweq{a}{L} \leq \pdim_\Q\beloweq{a}{L} \leq \pdim_\Q L \]
	for all $a \in L \setminus\set{\hat{1}}$.
	Hence iterating \Cref{lemma:reduction} yields that $\spdim_\Q L \leq \pdim_\Q L + 1$.
\end{proof}

\section{Applications}\label{sec:applications}
In this section, we give some additional applications of our results.

\subsection{The case \texorpdfstring{$k-2$}{k-2} and five generators}
\begin{theorem}\label{thm:k-2}
	Let $L$ be a maximal atomistic lattice with $k$ atoms and $\pdim_\Q L = k-2$.
	Then $\spdim_\Q L \leq k-2$ and $\spdim_\I L \leq k-1$.
\end{theorem}

\begin{corollary}\label{cor:k-2}
	Let $I \subset S$ be a monomial ideal with $k$ minimal generators.
	If $\pdim S/I \in \set{1,2, k-2,k-1,k}$, then the Stanley conjecture holds for $S/I$ and $I$.
\end{corollary}
\begin{proof}
	The cases $\pdim S/I = 1,2,k-1$ or $k$ are treated in \cite[Corollary 2.3]{HJY}, \cite[Lemma 4.3]{KSF}, \cite[Theorem 5.3]{lcm} and \cite[Proposition 5.2]{lcm}, respectively.
	The remaining case $k-2$ is immediate from the preceding theorem.
\end{proof}

\begin{corollary}\label{cor:five}
	Stanley conjecture holds for $S/I$ and $I$ for every monomial ideal $I$ with at most $5$ generators.
\end{corollary}

\begin{proof}[Proof of \Cref{thm:k-2}]
The Scarf complex $\Delta$ of $L$ is a stoss complex on $k$ vertices of dimension $k-3$.
So its Alexander dual $\Delta^\vee$ is a stoss complex of dimension $1$, i.e. a tree $\tilde{T}$.
We can choose a leaf $v$ of $\tilde{T}$, i.e. a vertex with only one edge $e$ attached to it.
Then $v$ satisfies the condition of \Cref{prop:zerlegungstoss} for $\Delta^\vee$ and thus also for $\Delta$.
Hence $L$ decomposes as $L = L_1 \# (L_1 \cap L_2)$, where
$L_1$ and $L_2$ are maximal lattices of projective dimension $k-2$ and $k-3$ on $k-1$ atoms.
By the description of the decomposition in \Cref{prop:zerlegungstoss}, we see that the Scarf complex of $L_2$ is again Alexander dual to a tree $T$, and that $T$ is obtained from $\tilde{T}$ by deleting $v$ and $e$.
Moreover, $L_1 = \bool{k-1} \setminus {w}$, where the element $w$ is the complement of the link of $v$ in $\tilde{T}$, i.e. the complement of the \emph{other} vertex of $e$.

Our strategy is to repeatedly remove elements of $L$ using \Cref{lemma:reduction} with $p = k-2$, until we are left with $L_2 \# L_2$.
Then it follows that 
\[\spdim_\epsilon L \leq \spdim_\epsilon L_2 \# L_2 \leq \spdim_\epsilon L_2 + 1 = \begin{cases}
(k-1)-2+1 &\text{ if }\epsilon = \Q,\\
(k-1)-1+1 &\text{ if }\epsilon = \I\\
\end{cases}\]
where we use \Cref{lemma:amal} and induction on $k$.

So consider the set $(L_1\# (L_1 \cap L_2)) \setminus (L_2 \# L_2) = L_1 \setminus L_2$.
It contains only elements of rank $k-3$ and $k-2$.
The elements of rank $k-3$ correspond to minimal non-faces of the Scarf complex of $L_2$, i.e. complements of edges of $T$.
On the other hand, elements of rank $k-2$ correspond to non-acyclic subcomplexes of the Scarf complex, i.e. complements of vertices of $T$, whose link is not acyclic, i.e. complements of non-leaf vertices of $T$.
In the sequel we will identify the elements of $L_1 \setminus L_2$ with the set of edges and non-leaf vertices of $T$.
We will remove the elements of $L_1 \setminus L_2$ from $L$ in the following order: 
First we remove all edges adjacent to $w$, then all non-leaf vertices adjacent to the previously removed edges,
then again all edges adjacent to previously removed vertices, and so on.
It is clear that after finitely many steps we reach $L_2\# L_2$.
Let us consider the steps more closely:
\begin{asparadesc}
\item[Edge step] The edge elements have rank $k-3$, so criterion (1) of \Cref{lemma:reduction} is satisfied. Moreover, an edge element $e$ is covered in $\bool{k}$ by its two vertices and the copy of the edge in $L_2$. But one of the vertices was removed before, and the copy is not contained in $L_2$ (by assumption), so $e$ is covered by only one element and hence meet-irreducible.

\item[Vertex step] Every vertex $v$ has rank $k-2$, and an edge $e$ below $v$ was removed before, so  criterion (2) of \Cref{lemma:reduction} is satisfied. Moreover, $v$ is meet-irreducible, since it is covered only by the maximal element of $L_1$.
\end{asparadesc}
\end{proof}


\subsection{Six generators}
\begin{theorem}\label{thm:six}
If $I \subset S$ is a monomial ideal with six minimal generators, then $I$ and $S/I$ satisfy the Stanley conjecture.
\end{theorem}
\begin{figure}
\newcommand{\crossnode}[1]{
\draw[draw=gray] (#1) +(-0.15,-0.15) -- +(0.15,0.15)  +(0.15,-0.15) -- +(-0.15,0.15);
}
\newcommand{\crossedge}[2]{
\crossnode{$(#1)!.5!(#2)$}
}
\begin{tikzpicture}[scale=1, every node/.style={circle,inner sep=0pt, fill=black,  minimum size=1.5mm, draw}]
	\begin{scope}
		\path
			node (1) at (0,0) {}
			node (2) at (0:1) {}
			node (3) at (90:1) {}
			node (4) at (180:1) {}
			node (5) at (270:1) {};
			
		\draw (1) -- (2);
		\draw (1) -- (3);
		\draw (1) -- (4);
		\draw (1) -- (5);
		
		\crossnode{2}
		\crossnode{3}
		\crossnode{4}
		\crossnode{5}
	\end{scope}
	
	\begin{scope}[xshift=3cm]
		\path(0,0)
			node (1) at (0,0) {}
			node (2) at (120:1) {}
			node (3) at (240:1) {}
			node (4) at (1,0) {}
			node (5) at (2,0) {};

		\draw (1) -- (2);
		\draw (1) -- (3);
		\draw (1) -- (4) --(5);
		
		\crossnode{2}
		\crossnode{3}
		\crossnode{4}
		\crossnode{5}

		\crossedge{4}{5}
	\end{scope}
	
	\begin{scope}[xshift=7cm]
		\path(0,0)
			node (1) at (0,0) {}
			node (2) at (1,0) {}
			node (3) at (2,0) {}
			node (4) at (3,0) {}
			node (5) at (4,0) {};
			
		\draw (1) --(2) --(3) -- (4) -- (5);
		
		\crossnode{1}
		\crossnode{2}
		\crossnode{4}
		\crossnode{5}

		\crossedge{1}{2}
		\crossedge{4}{5}
	\end{scope}
\end{tikzpicture}
\caption{The trees with five vertices.}
\label{fig:trees5}
\end{figure}
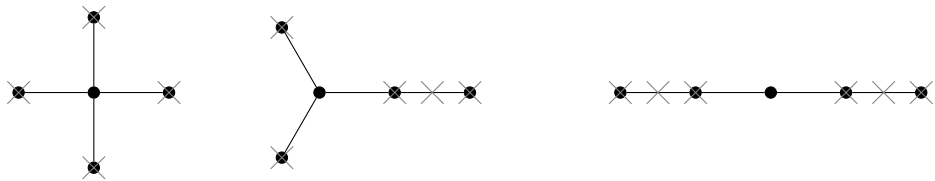
\begin{proof}
We only need to consider maximal atomistic lattices $L$ with six atoms.
Also, we may assume that $\KK = \QQ$.
Indeed, the Stanley projective dimension does not depend on the characteristic.
Moreover, recall that a simplicial complex is $\QQ$-acyclic if it is acyclic over any field.
Hence a $\KK$-stoss complex for any field is also a $\QQ$-stoss complex.
So if we compute the Stanley projective dimension for every $\QQ$-stoss complex, then we have in particular computed it for every $\KK$-stoss complex for any field $\KK$.
Moreover, we only need to consider the case $\pdim_\Q L = 3$ by \Cref{cor:k-2}.

First, assume that $L$ is an amalgamation $L_1\#(L_1 \cap L_2)$.
Then by \Cref{prop:zerlegungstoss} $L_1$ and $L_2$ are maximal lattices on five atoms of projective dimension $3$ and $2$.
As noted before, the Scarf complex of $L_2$ is a tree $T$ on five vertices.
As in \Cref{ex:tree}, we can use \Cref{lemma:reduction} to remove all leaves of $T$.
However, we cannot restrict the resulting lattice $L'$ to its atomistic sublattice,
because it might still be atomistic inside $L_1 \# L'$.
So consider a leaf $v$ of $T$ and let $w$ be its unique neighbor.
The edge $v\vee w$ is an element of rank $3$, which does not lie in the Scarf complex (after removing $v$),
so we may remove it if it is meet-irreducible.
However, in general this is not the case.
But if $w$ has degree $2$ in $T$ (i.e. $w$ has only one further neighbor, say $u$),
then every connected induced subcomplex of $T$ properly containing $\set{v,w}$ also contains $\set{w,u}$ and hence $v \vee w \vee u$ is the unique element covering $v \vee w$.
So in this case, we may remove $v \vee w$.
Moreover, this renders $w$ meet-irreducible, so we may remove it as well.

There are only three trees with five vertices, see \Cref{fig:trees5}.
We have crossed out the removed vertices and edges.
By direct inspection, one sees that in each case only one vertex survives.
Moreover, in every case this vertex is contained in every non-crossed connected induced subcomplex.
Hence the minimal element of $L_2$ inside $L$ is covered by only this vertex and is thus meet-irreducible.
So we may remove this element and obtain a lattice with only five atoms (and projective dimension at most $3$).
But we already know that every lattice on five atoms satisfies the Stanley conjecture, so we are done.

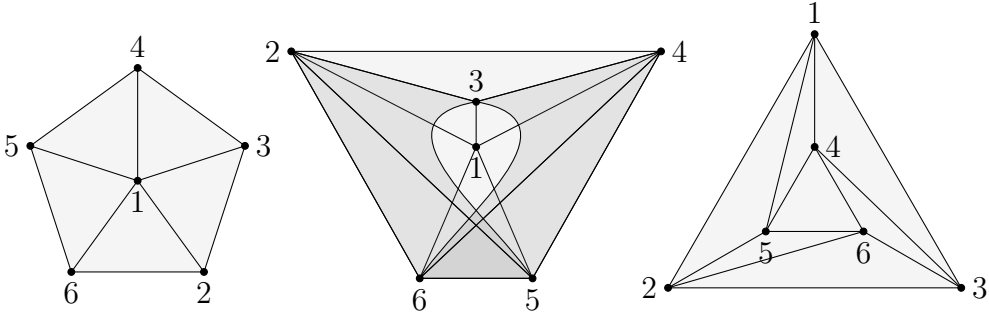
\begin{figure}
\begin{tikzpicture}[scale=1.5]
\begin{scope}[rotate=-90]
	\coordinate (m) at (0,0);
	\coordinate (v1) at (36:1);
	\coordinate (v2) at (108:1);
	\coordinate (v3) at (180:1);
	\coordinate (v4) at (252:1);
	\coordinate (v5) at (324:1);

	\filldraw[fill opacity=0.15,fill=gray!50] (v1)--(v2)--(v3)--(v4)--(v5)--cycle;
	\foreach \p in {(v1),(v2),(v3),(v4),(v5)}
		\draw (m) -- \p;

	\foreach \a/\b/\c in {(v1)/(v2)/(v3),(v2)/(v3)/(v4),(v3)/(v4)/(v5),(v4)/(v5)/(v1),(v5)/(v1)/(v2)}
		\filldraw[fill opacity=0.15,fill=gray!100] \a -- \b -- \c -- cycle;

	\foreach \p in {(m),(v1),(v2),(v3),(v4),(v5)}
		\draw[fill=black] \p circle (0.03);

	\path (m) node[anchor=north] {$1$};
	\path (v1) node[anchor=north] {$2$};
	\path (v2) node[anchor=west] {$3$};
	\path (v3) node[anchor=south] {$4$};
	\path (v4) node[anchor=east] {$5$};
	\path (v5) node[anchor=north] {$6$};
\end{scope}

\begin{scope}[xshift=3cm,scale=1]
	\coordinate (v1) at (0,0.3);
	\coordinate (v2) at (145:2);
	\coordinate (v3) at (90:0.7);
	\coordinate (v4) at (35:2);
	\coordinate (v5) at (300:1);
	\coordinate (v6) at (240:1);

	\filldraw[fill opacity=0.15,fill=gray!50] (v2)--(v4)--(v5)--(v6) -- cycle;

	\foreach \p in {(v2),(v3),(v4),(v5),(v6)}
		\draw (v1) -- \p;
	\draw (v2) -- (v3) -- (v4);

	\foreach \a in {(v2),(v4)}
		\filldraw[fill opacity=0.15,fill=gray!100] \a -- (v5) -- (v6) -- cycle;

	\filldraw[fill opacity=0.15,fill=gray!100] (v6) .. controls ($(v6)!0.2!15:(v4)$) and ($(v3)!0.6!-25:(v4)$)  .. (v3) -- (v4) -- cycle;
	\filldraw[fill opacity=0.15,fill=gray!100] (v5) .. controls ($(v5)!0.2!-15:(v2)$) and ($(v3)!0.6!25:(v2)$)  .. (v3) -- (v2) -- cycle;

	\foreach \p in {(v1),(v2),(v3),(v4),(v5),(v6)}
		\draw[fill=black] \p circle (0.03);

	\path (v1) node[anchor=north] {$1$};
	\path (v2) node[anchor=east] {$2$};
	\path (v3) node[anchor=south] {$3$};
	\path (v4) node[anchor=west] {$4$};
	\path (v5) node[anchor=north] {$5$};
	\path (v6) node[anchor=north] {$6$};
\end{scope}

\begin{scope}[xshift=6cm,scale=1,yshift=-0.2cm]
	\coordinate (v1) at (90:1.5);
	\coordinate (v2) at (210:1.5);
	\coordinate (v3) at (330:1.5);
	\coordinate (v4) at (90:0.5);
	\coordinate (v5) at (210:0.5);
	\coordinate (v6) at (330:0.5);

	\filldraw[fill opacity=0.15,fill=gray!50] (v1)--(v2)--(v3) -- cycle;

	\foreach \p in {(v1),(v3),(v5),(v6)}
		\draw (v4) -- \p;
	\foreach \p in {(v1),(v2),(v6)}
		\draw (v5) -- \p;
	\foreach \p in {(v2),(v3)}
		\draw (v6) -- \p;

	\foreach \a/\b/\c in {(v1)/(v3)/(v6),(v1)/(v2)/(v4),(v2)/(v3)/(v5)}
		\filldraw[fill opacity=0.15,fill=gray!100] \a -- \b -- \c -- cycle;

	\foreach \p in {(v1),(v2),(v3),(v4),(v5),(v6)}
		\draw[fill=black] \p circle (0.03);

	\path (v1) node[anchor=south] {$1$};
	\path (v2) node[anchor=east] {$2$};
	\path (v3) node[anchor=west] {$3$};
	\path (v4) node[anchor=west] {$4$};
	\path (v5) node[anchor=north] {$5$};
	\path (v6) node[anchor=north] {$6$};
\end{scope}
\end{tikzpicture}
\caption{The complexes $C_1, C_2$ and $C_3$ from the proof of \Cref{thm:six}.}
\label{fig:c1c2c3}
\end{figure}
It remains to consider the case that $L$ is not the amalgamation of two smaller lattices. 
In this case, the Scarf complex of $L$ is a $2$-dimensional stoss complex on six vertices, such that every vertex is contained in at least $5$ triangles.
It has been determined by Kalai in \cite{kalai} that there are only four complexes of this type.
First, consider the three complexes named $C_1, C_2$ and $C_3$ in \cite[Table 1]{kalai}, see \Cref{fig:c1c2c3}.
For these complexes, one can bound the Stanley projective dimension using just \Cref{lemma:reduction} several times,
so the corresponding lattices satisfy the Stanley conjecture.
In the table below, we give the facets of the complexes and the elements of $\exlat \setminus \Delta$.
The column labeled \qq{deletion order} lists elements $a$, which are to be removed by \Cref{lemma:reduction} (in the given order).

\begin{tabular}{|c|x{3.5cm}|x{3.7cm}|x{3cm}|}
\hline \rule[-2ex]{0pt}{5.5ex} Name & Facets & $\exlat \setminus \Delta$ & Deletion order \tabularnewline 
\hline \rule[-2ex]{0pt}{5.5ex} $C_1$& $123\ 134\ 145\ 156\ 126$ $234\ 345\ 456\ 256\ 236$ & $1234\ 1345\ 1456$ $1562\ 1623\ 123456$ & $24\ 234\ 25\ 256$ $36\ 236\ 23\ 26\ 2$ \tabularnewline 
\hline \rule[-2ex]{0pt}{5.5ex} $C_2$& $123\ 134\ 145\ 156\ 126$ $234\ 235\ 256\ 346\ 456$ & $1234\ 1256\ 1456$ $123456$ & $35\ 235\ 36\ 346$ $24\ 234\ 23\ 34\ 3$ \tabularnewline 
\hline \rule[-2ex]{0pt}{5.5ex} $C_3$& $124\ 125\ 134\ 145\ 136$ $235\ 236\ 256\ 346\ 456$ & $1245\ 2356\ 1346$ $123456$ & $16\ 136\ 13\ 134\ 24$ $124\ 12\ 125\ 15\ 5$ \tabularnewline 
\hline 
\end{tabular}

Note that the rank assumption of \Cref{lemma:reduction} is satisfied automatically for the elements of rank $2$. Further, for each element of rank $3$ which is to be removed, a  rank $2$ element below it was removed earlier, so condition (2) of \Cref{lemma:reduction} applies.
It remains to check that every element to be removed is covered by exactly one other element. This can be done by inspection.
We stop once we removed an atom, because then we have reduced the situation to the case of five generators, so the claim follows from \Cref{cor:five}.

The last case of \cite[Table 1]{kalai} is $P_2$, which is the six-vertex triangulation of the real projective plane.
For $\spdim_\I L$, last line of \Cref{lemma:reduction} is still applicable to all elements of rank $3$, so it follows that $\spdim_\I L \leq 2$.
For $\spdim_\Q L$, we resort to a computational proof.
We verified that $\spdim_\Q \leq 3$ using the method described in \Cref{ssec:IP}.
The actual computation using \texttt{SCIP} took less than a second.
\end{proof}

\begin{remark}
\begin{enumerate}
	\item An ideal with a given lcm lattice can be constructed using Theorem 3.4 in \cite{lcm} or Theorem 3.2 in \cite{M}.
	In the case $P_2$ in the preceding proof, it turns out that one can choose the \emph{nearly Scarf ideal} \cite{PV}
	associated to the $6$-vertex projective plane. 
	In general, the extremal ideals associated to a stoss complex $\Delta$ are not the nearly Scarf ideals.
	Indeed, the lcm lattice of the latter equals $\Delta$ with an additional maximal element; while the lcm lattice of the former can contain many additional elements according to \Cref{thm:reconstruction}.
	
	\item The cautious reader might notice that the last facet of $C_3$ in \cite[Table 1]{kalai} is $356$ instead of $456$. This is indeed a small misprint in \cite{kalai}.
	
	\item The reason why we only give a computer proof for  the case $P_2$ is the following.
	Every edge of this complex is contained in two triangles and is thus not meet-irreducible, and every element of rank $3$ is in the Scarf complex.
	So we cannot apply \Cref{lemma:reduction} for $\spdim_\Q L$.
	We also tried to compute a Stanley decomposition using an implementation of the algorithm given in \cite{IZ}, but unfortunately this seems infeasible.
\end{enumerate}
\end{remark}

\subsubsection{Computing Stanley decompositions via linear Diophantine equations}\label{ssec:IP}
\newcommand{\supp}{\mathop{\mathrm{supp}}}
\newcommand{\af}{\mathbf{a}}
\newcommand{\bF}{\mathbf{b}}
\newcommand{\gf}{\mathbf{g}}
\newcommand{\Tf}{\mathbf{T}}
Let us for the moment return to a more general setup.
The following approach for computing the Stanley depth was suggested by Winfried Bruns.
Let $M$ be a finitely generated multigraded $S$-module.
Recall that the (multigraded) \emph{Hilbert series} of $M$ is the formal power series
\[
	H_M(\Tf) = \sum_{\mathbf{a} \in \NN^n} (\dim_\KK M_{\af}) \Tf^\mathbf{a}
\]
where $\Tf^\af = T_1^{a_1}\dotsm T_n^{a_n}$ as usual.
A \emph{Hilbert decomposition} is a finite family $(I_i, \af_i)_{i \in \mathcal{I}}$, where $I_i \subset \set{1,2,\dotsc,n}$ are subsets and $\mathbf{a}_i \in \NN^n$, such that
\begin{equation}\label{eq:hilbdec}
H_M(\Tf) = \sum_{i \in \mathcal{I}} \Tf^{\af_i} \prod_{j \in I_i} \frac{1}{1-T_j}.
\end{equation}
The \emph{Hilbert depth} of such a decomposition is the minimal cardinality of the $I_i$, 
and the Hilbert depth of $M$ is the maximum of the Hilbert depths over all Hilbert decompositions of $M$.

It follows from \cite[Proposition 2.8]{bku} that if $M = S/I$ or $M=I$, then the Stanley and Hilbert depth coincide.
So for the case we are interested in we only need to compute the Hilbert depth.

Let $\gf \in \NN^n$ be a multidegree such that $M$ is $\gf$-determined in the sense of \cite{miller}.
In particular, if $M = S/I$ or $M=I$, we may take the lcm of all minimal monomial generators of $I$ as $\gf$.
It has been shown in \cite{HVZ} that one only needs to consider Hilbert decomposition with $\af_i \preceq \gf$ (where $\preceq$ stands for the componentwise order) and that one only needs to check \eqref{eq:hilbdec} in the degrees $\preceq \gf$.
For $\bF \in \NN^n$, we set $\bF_\mathrm{max} := \set{j \with \bF_j = \gf_j}$ and $\supp \bF := \set{j \with \bF_j \neq 0}$.
It follows from \cite[Theorem 3.3]{IM} that every Hilbert decomposition satisfies $(\af_i)_\mathrm{max} \subseteq I_i$ for all $i$.

We make the following ansatz for \eqref{eq:hilbdec}:
\[
	\sum_{0 \preceq \bF \preceq \gf} \sum_{\substack{F \in \mathcal{F} \\ \bF_\mathrm{max} \subseteq F}} c_{F,\bF} \Tf^{\bF} \prod_{j \in F} \frac{1}{1-T_i}
\]
with unknown coefficients $c_{F,\bF}$. Here $\mathcal{F}$ is the set of possible \qq{building blocks} in the Hilbert decomposition.
So for showing that $M$ has Hilbert depth at least $h$, one can choose $\mathcal{F}$ to be the set of all subsets of $\set{1,\dotsc,n}$ of cardinality at least $h$.
Comparing coefficients with the Hilbert series of $M$ one obtains the following system of Diophantine equations and inequalities:
\begin{equation}
\begin{aligned}\label{eq:dio}
\dim_\KK M_\af &= \sum_{0 \preceq \bF \preceq \af} \sum_{\substack{F \in \mathcal{F} \\ \mathbf{b}_\textrm{max} \subseteq F \\
\supp (\af - \bF) \subseteq F}} c_{F,\bF} &\text{ for } 0 \preceq \af \preceq \gf \\
c_{F,\bF} &\geq 0 &\text{ for } F \in \mathcal{F}, \bF_\textrm{max} \subseteq F \\
c_{F,\bF} &\in \ZZ &\text{ for } F \in \mathcal{F}, \bF_\textrm{max} \subseteq F
\end{aligned}
\end{equation}
So for proving that $M$ has Hilbert depth at least $h$, one needs to show that the system \eqref{eq:dio} has a solution.
We implemented this in \texttt{SCIP} \cite{SCIP}.
\begin{remark}
	Let us mention some reductions to make the problem computationally easier.
	By \cite[Theorem 12]{IZ}, one only needs to consider sets of cardinality \emph{exactly} $h$ for $\mathcal{F}$.
	Moreover, one clearly has $c_{F, \bF} = 0$ for all $F$ if $M_\bF = 0$. 
\end{remark}

\subsection{Seven generators}
The situation is even more complicated if we consider ideals with seven generators.
For ideals, we were able to obtain a complete answer:
\begin{theorem}\label{thm:sevenI}
If $I \subset S$ is a monomial ideal with seven minimal generators. Then $I$ satisfies the Stanley conjecture.
\end{theorem}

On the other hand, for quotients of ideals, we only get the following partial result:
\begin{proposition}\label{prop:sevenP}
There is an explicit list of $211$ monomial ideals with seven generators with the following property:
If $\spdim S/I \leq \pdim S/I$ for all ideals in this list, then the Stanley conjecture holds for all quotients by ideals with up to seven generators.

Moreover, the Stanley conjecture holds unconditionally for all quotients by ideals with up to seven generators in characteristic $2$.
\end{proposition}
The list is available from the author upon request.
These two results are obtained by a rather extensive computer search.
We describe now how we proceeded.
As in the proof of \Cref{thm:six} we may assume that $\KK = \QQ$ for the enumeration of stoss complexes.
\begin{proof}[Proof of \Cref{thm:sevenI}]
	We only need to consider the case $\pdim S/I \in \set{3,4}$ by \Cref{cor:k-2}.
	Moreover, for every ideal with seven generators, we have $\spdim I \leq \lfloor\frac{7}{2}\rfloor = 3$.
	Hence if $\pdim S/I = 4$, then $\pdim I = 3 \geq \spdim I$, so in this case the Stanley conjecture holds.
	
	It remains the case $\pdim S/I = 3$.
	We are going to completely enumerate all maximal lattices of projective dimension $3$ on $7$ atoms.
	First consider amalgamations of lattices.
	Every such lattice comes from a $2$-dimensional stoss complex on seven vertices, which by \Cref{prop:zerlegungstoss} can be
	decomposed into a $2$-dimensional stoss complex on $6$ vertices and a tree on $6$ vertices.
	By symmetry, we may assume that the vertex $v$ of \Cref{prop:zerlegungstoss} is vertex number seven.
	Again, by symmetry, we only need to consider one representative for every isomorphism class of stoss complexes of projective dimension $(3-1)$ on $6$ vertices; there are $84$.
	On the other hand, we need to consider every tree on $6$ vertices; there are $6^{6-2} = 1296$ by the Matrix-Tree theorem.
	So we obtain $84 \cdot 1296 = 108864$ stoss complexes.
	When we pick one representative per isomorphism class we are left with $50651$ complexes.
	For every one we verified the Stanley conjecture using \Cref{lemma:reduction}.
	
	Next we consider those maximal lattices which are not amalgamations.
	Let $\Delta$ be the Scarf complex of such a lattice.
	Every vertex is contained in at least $\binom{7-2}{3-2} + 1 = 6$ facets.
	In fact, by the following counting argument we see that there has to be a vertex $v$ with exactly this number.
	We count the number of vertex-facet incidences. There are $\binom{7-1}{3-1} = 15$ facets and each has $3$ vertices, so there are $45$ incidences.
	On the other hand, there are $7$ vertices, so if each one is contained in at least $7$ facets, we obtain at least $49$ incidences, a contradiction.

	Let $v$ be a vertex which is contained in exactly six facets.
	Then $\lk_\Delta v$ is a connected graph on six vertices, which has one more edge than every tree on six vertices.
	So we can consider $\lk_\Delta v$ as a tree with an additional edge.
	Further, $\del_\Delta v$ is a $2$-dimensional simplicial complex with a complete $1$-skeleton, whose top boundary map is injective (as it is a restriction of the boundary map of $\Delta$).
	As the number of facets of $\del_\Delta v$ is one less than the number of facets of a stoss complex (i.e. the dimension of the kernel of the $1$-dim boundary map), we can add a facet to kill the remaining cycle and to turn $\del_\Delta v$ into a stoss complex.
	Hence, $\del_\Delta v$ can be considered as a $2$-dimensional stoss complex where one facet is missing.

	We enumerate all complexes obtainable by removing one facet from a $2$-dimensional stoss complex on $6$ vertices and filter by isomorphism; there are $234$ isomorphism classes.
	Then we enumerate all graphs on $6$ vertices having $6$ edges and containing a tree; there are $3660$ of them.
	We combine these using $\Delta = \del_\Delta v \cup v * \lk_\Delta v$.
	Not every complex with a link and deletion of this kind is really a stoss complex, but it will be a $2$-dimensional complex with a complete $1$-skeleton. So it is sufficient to check that the complex is acyclic.
	In fact, as the complex has the \qq{expected} number of facets, we only need to verify that the top boundary matrix has full rank.
	After this check and after filtering by isomorphism, there remain $9726$ complexes.
	Again, we were able to verify the Stanley conjecture using \Cref{lemma:reduction} for each of them.
\end{proof}
\begin{proof}[Proof of \Cref{prop:sevenP}]
	Here we have to consider both cases $\pdim S/I = 3$ and $\pdim S/I = 4$.

	Of the $50651$ $2$-dimensional stoss complexes which come from amalgamation, we could verify the Stanley conjecture by \Cref{lemma:reduction} for all but $25$ cases. These cases, however, get reduced by \Cref{lemma:reduction} to the single exceptional case  $P_2$ in $6$ generators from the proof of \Cref{thm:six}, so we know from the computation there that the Stanley conjecture holds in this case.
	For the $9726$ $2$-dimensional complexes that are no amalgamations, \Cref{lemma:reduction} works in all but $93$ cases. These cases remain open so far.
	
	The $3$-dimensional stoss complexes on $7$ vertices are the Alexander duals of the $2$-dimensional ones, so we do not need to do a new enumeration.
	Moreover, Alexander duality respects the property of being an amalgamation.
	Again, of the $50651$ $3$-dimensional stoss complexes which come from amalgamation, we could verify the Stanley conjecture by \Cref{lemma:reduction} for all but $25$ cases. These are in fact the Alexander duals of the difficult $2$-dimensional cases.
	However, in this case they do not reduce to something we know, so these cases remain open as well.
	Finally, for the $9726$ $3$-dimensional complexes that are no amalgamations, \Cref{lemma:reduction} works again in all but $93$ cases, and these are again the Alexander duals of the difficult cases form above. They remain open as well.
	
	Altogether, there remain $93+25+93 = 211$ open cases.
	In all these cases, the Scarf complex is not acyclic in characteristic $2$, so under this assumption we may ignore these cases.
\end{proof}
\begin{remark}
	Unfortunately, the open cases from the preceding proof seem to be too big to be solved using the method described in \ref{ssec:IP} above.
\end{remark}
\begin{remark}
	We would like to mention that there are several interesting articles by D. Popescu and his coauthors \cite{popescu2014, popescuFour, popescuFive} concerning the Stanley depth of monomial ideals under assumptions on the minimal generators.
	However, these results are quite different in nature from the results of the present paper.
	So our results neither imply nor are implied by the results of \cite{popescu2014, popescuFour, popescuFive}.
\end{remark}

\section{Discussion and open problems}
In this section, we discuss some open problems and directions for further research.

In the present paper, we have reduced the Stanley conjecture to a rather narrow class of ideals.
Recall that for every maximal lattice, one can choose a squarefree ideal realizing it. This ideal can be considered as the Stanley-Reisner ideal of a certain simplicial complex, and a Stanley decomposition of the ideal corresponds to a partition of that complex.
So it is natural to ask for the simplicial complexes arising in this way.
\begin{question}
What can be said about these complexes? 
\end{question}
It follows from \Cref{cor:extremal} that these complexes are acyclic.
However, it is not so clear what further information one can expect.
For a start, a simplicial complex and its one-point suspension have the same lcm lattice, so it only makes sense to ask about properties that are preserved under suspension.

Another natural variation of the questions considered in the present paper is to consider \emph{minimal} lattices with a given projective dimension.
As first guess, one might expect that their crosscut complexes should be homology spheres, but this might be very wrong in higher dimension.
A result in this direction could, for example, be used to study the following question.
It is motivated by the observation in \cite{IKM3} that the Stanley projective dimension and the (usual) projective dimension coincide for all ideals with up to five generators.
\begin{question}
Let $I \subseteq S$ be an ideal with $k$ minimal generators, such that $\pdim S/I \in \set{1,2,k-2,k-1,k}$. Does this imply that $\spdim S/I = \pdim S/I$?
\end{question}
For $\pdim S/I = 1$ or $k$ this is easy.
For $\pdim S/I = 2$, the corresponding \emph{minimal} lattice is the boolean lattice on $2$ atoms, so the question also has a positive answer.
Thus only $k-1$ and $k-2$ are open.

In view of \Cref{lemma:amal} and \Cref{thm:amalpdim} we offer the following conjecture.
\begin{conjecture}\label{conj:amal}
	Let $L_2 \subseteq L_1$ be two (finite atomistic) lattices.
	If $\spdim_\Q L_2 < \spdim_\Q L_1$, then $\spdim_\Q L_1 \# L_2 = \spdim_\Q L_1$. 
	In other words, the lower bound in \Cref{lemma:amal} is always attained.
\end{conjecture}
This conjecture is also interesting if one assumes that the lattices are maximal.
Example \ref{ex:tree} and the proofs of \Cref{thm:k-2} and \Cref{thm:six} can be seen as special cases of this conjecture.
If it is true, then one could reduce the Stanley conjecture to those maximal lattices that are not amalgamations of smaller maximal lattices.

Finally, we would like to mention that \Cref{lemma:reduction} seems very powerful for computing the Stanley projective dimension. 
With the aid of a computer, we have verified that one can compute the Stanley projective dimension of all ideals with up to five generators just using this lemma. Even for six generators it seems to work for the vast majority of cases.
On the other hand, there is at least one systematic failure of this approach. Namely, it is not difficult to see that \Cref{lemma:reduction} will also give an upper bound for the projective dimension, and this bound is characteristic free. 
Hence if $I$ is an ideal such that $\pdim_{\Char 0} S/I < \pdim_{\Char p} S/I$, then \Cref{lemma:reduction} is very unlikely prove that $\spdim_{\Char 0} S/I \leq \pdim_{\Char 0} S/I$. 
This happens for example with the Stanley-Reisner ideals of projective spaces. So in particular the exceptional case in the proof of \Cref{thm:six} is of this type.
We wonder if one could somehow improve \Cref{lemma:reduction} to avoid this problem.

On the other hand, one can always use \Cref{lemma:reduction} as a preprocessing step in computing the Stanley depth. 
So the following approach seems promising to us: For a given ideal, one can compute its lcm lattice, then reduce it with \Cref{lemma:reduction}, then chose an ideal for the reduced lattice and compute its Stanley depth by other means.
The choice of the ideal can be easily done by \cite[Theorem 3.4]{lcm} or \cite[Theorem 3.2]{M} and there are well-known algorithms for computing the Stanley depth of ideals \cite{HVZ,IZ}.

\section*{Acknowledgment}
The author would like to thank Winfried Bruns for suggesting the method described in \ref{ssec:IP} and Markus Spitzweck for discussing the proof of \Cref{thm:extremal}.
Moreover, the author thanks Bogdan Ichim, Francesco Strazzanti, Richard Sieg and Mihai Cipu for pointing out many typos in an earlier version of this article.
Further, I would like to thank the two anonymous reviewers for several helpful hints and suggestions.

\printbibliography

\end{document}